\documentclass[12pt]{amsart}
\usepackage[english]{babel}
\usepackage{amsmath}
\usepackage{changes}
\usepackage{amsthm}
\usepackage{amsfonts} 
\usepackage{hyperref}
\usepackage{epsfig}
\usepackage{amssymb}
\usepackage{mathrsfs}
\DeclareMathOperator*{\argmin}{arg\,min}

\numberwithin{equation}{section}

\allowdisplaybreaks

\def\parent#1{#1^{-1}}
\def\r{\varrho}

\newcommand{\eqa}{\begin{eqnarray}}
\newcommand{\ena}{\end{eqnarray}}
\newcommand{\eq}{\begin{equation}}
\newcommand{\en}{\end{equation}}
\newcommand{\eqs}{\begin{eqnarray*}}
\newcommand{\ens}{\end{eqnarray*}}

\def\r{\varrho} 

\newcommand{\Z}     {\mathbb{Z}} 
\newcommand{\N}     {\mathbb{N}} 
\renewcommand{\P}   {\mathbb{P}} 
\newcommand{\E}     {\mathbb{E}} 
\newcommand{\Go}     {{\scriptsize RWRC}}
\newcommand{\Or}     {{\scriptsize ORRW}}

\def\1{{\mathchoice {1\mskip-4mu\mathrm l}      
{1\mskip-4mu\mathrm l} 
{1\mskip-4.5mu\mathrm l} {1\mskip-5mu\mathrm l}}} 
\newcommand{\ssup}[1] {{{\scriptscriptstyle{({#1}})}}} 
\def\comment#1{} 
 
\renewcommand{\d}{{\rm d}} 
 
\newcommand{\eps}{\varepsilon}

\newcommand{\Ccal}   {{\mathcal C }}

\newcommand{\Tcal}   {{\mathcal T }} 

\def\ignore#1{}

\def\Def{\ :=\ }

\def\C{\mathcal{C}}

\def\bP{\mathbf{P}}

\def\bE{\mathbf{E}}

\def\parent#1{#1^{-1}}

\newtheorem{theorem}{Theorem}
\newtheorem{proposition}[theorem]{Proposition}
\newtheorem{lemma}[theorem]{Lemma}
\newtheorem{remark}[theorem]{Remark}

\newtheorem{definition}[theorem]{Definition}
\newtheorem{corollary}[theorem]{Corollary}

\renewcommand{\epsilon}{\varepsilon}

\title[The branching-ruin number]{The branching-ruin number as critical parameter of random processes on trees
 }
\date{}
\author[A.~Collevecchio]{Andrea Collevecchio}
\address{Andrea Collevecchio\\ School of Mathematical Sciences, Monash  University, Melbourne} \email{andrea.Collevecchio@monash.edu}
\author[C.-B.~Huynh]{Cong Bang Huynh}
\address{Cong Bang Huynh\\ Univ. Grenoble Alpes, CNRS, Institut Fourier, F-38000 Grenoble, France}
\email{cong-bang.huynh@univ-grenoble-alpes.fr}
\author[D.~Kious]{Daniel Kious}
\address{Daniel Kious\\ NYU-ECNU Institute of Mathematical Sciences at NYU Shanghai} \email{daniel.kious@nyu.edu}


\begin{document}

\begin{abstract}
The {\it branching-ruin number} of a tree, which describes its asymptotic growth and geometry, can be seen as a polynomial version of the branching number. This quantity was  defined by Collevecchio, Kious and Sidoravicius (2018) in order to  understand the phase transitions of the once-reinforced random walk ({\Or }) on trees. Strikingly, this number was proved to be equal to the critical parameter of {\Or } on trees.\\
In this paper, we continue the investigation of the link between the branching-ruin number and the criticality of random processes on trees.\\
First, we study random walks on random conductances on trees, when the conductances have an heavy tail at $0$, parametrized by some $p>1$, where $1/p$ is the exponent of the tail. We prove a phase transition recurrence/transience with respect to $p$ and identify the critical parameter to be equal to the branching-ruin number of the tree.\\
Second, we study a multi-excited random walk on trees where each vertex has $M$ cookies and each cookie has an infinite strength towards the root. Here again, we prove a phase transition recurrence/transience and identify the critical number of cookies to be equal to the branching-ruin number of the tree, minus 1. This result extends a conjecture of  Volkov (2003). Besides, we study a generalized version of this process and generalize results of Basdevant and Singh (2009).
\end{abstract}

\maketitle
\section{Introduction}
Let us consider a random process on a tree which is parametrized with one parameter $p$. We say that this process undergoes a {\it phase transition}  if there exists a {\it critical parameter} $p_c$ such that the (macroscopic) behavior of the random process is significantly different for $p<p_c$ and for $p>p_c$. This is, for instance, the case of Bernoulli percolation on trees, biased random walks (see \cite{L90,LyoPem,LP})    or linearly edge-reinforced random walks \cite{Pemtree} on trees.\\
In \cite{L90}, R.~Lyons {proved the  following beautiful result.}  {Bernoulli percolation and biased random walks (among others) share  the  same critical parameter which is equal to the {\it branching number} of the tree.} The branching number, defined by Furstenberg \cite{Furs}, is, roughly speaking, a quantity that provides a precise information on the asymptotic growth and geometry of a tree, at the exponential scale (see \eqref{branchingdef} for a definition). For instance, for trees that are ``well-behaved'' ({such as  spherically symmetric trees}) and whose spheres of diameter $n$ have size $m^n$, the branching number is equal to $m$. {This description is actually not accurate as some trees have a peculiar geometry, and the size of their spheres is not a good indicator of their asymptotic complexity.}\\
The phase transition of the once-reinforced random walk was studied in \cite{CKS}. In order to see a phase transition, one needs to consider trees that grow polynomially fast (see \cite{KS16}), and therefore the branching number is not the quantity that would provide a relevant information in this case. Indeed, the branching number does not allow us to distinguish among trees with polynomial growth as the branching number of {\it any} tree with sub-exponential growth is equal to 1. In \cite{CKS}, it was proved that the critical parameter for the once-reinforced random walk on trees is equal to the {\it branching-ruin number} of the tree (see \eqref{branchingdef2}). The branching-ruin number of a tree is best described as the polynomial version of the branching number: if a well-behaved tree has spheres of size $n^b$, then the branching-ruin number of this tree is $b$. Again, this fact is not true in general because of the possible complex asymptotic geometry of trees.\\
The purpose of the current paper is to emphasize two other examples where the branching-ruin number appears as the critical parameter of a random process, as it was done for the branching number. We study random walks on random conductances with heavy-tails and a model of excited random walks called the $M$-digging random walk.
In the next two subsections, we describe our results. In the first one, we relate the branching-ruin number to the critical weight of the tails of the conductances. In the second result, we relate the critical number of cookies per site to the branching-ruin number and, in particular, our result extends a conjecture of Volkov \cite{Volk}.


\subsection{Random walk on heavy-tailed random conductances}  
First, we study random walks on random conductances in the case where the  conductances {have heavy tails at zero}.  Consider an infinite, locally finite, tree $\mathcal{T}$ with branching-ruin number $b$ (see \eqref{branchingdef2} for a definition). {Even though our results hold for any branching-ruin number, for the sake of the following explanations, let us temporarily assume that $b>1$, so that simple random walk is transient on this tree (see Theorem \ref{th:digging}, or \cite{CKS})}.
Assign i.i.d.~conductances, or weights, to each edge of $\mathcal{T}$ and let us define a nearest-neighbor random walk which jumps through an edge with a probability proportional to the conductance of this edge. This model is very classical and has been extensively study on various graph, including $\mathbb{Z}$ and $\mathbb{Z}^d$. The behavior of the walk depends on the common law of the conductances.\\
For instance, if the conductances are bounded away from $0$ and from the infinity, the behavior of the walk is close to the one of simple random walk and it will therefore be transient on $\mathcal{T}$, moving at a speed similar to that of simple random walk.\\
If the conductances can be very large, i.e.~unbounded and for instance with an heavy-tail at infinity, this should not affect the transience of the walk. Nevertheless, this would have an important impact on the time that the random walk spends on small areas of the environment. We do not prove anything in this direction in this paper as our main interest is in the recurrence/transience of the walk, but we would like to describe here what should happen. If the conductances can be extremely large with a not-so-small probability, then the walker will meet, here and there, an edge with an overwhelmingly large conductance and will cross this edge back-and-forth for a very large number of times before moving on. The consequence of this mechanism is that the random walker will spend most of its time on these {\it traps} and will move at a speed much smaller than simple random walk on the same tree. This phenomenon is reminiscent of Bouchaud's trap model, see \cite{fin02,ESZ,ESZ1,ESZ2}, or \cite{FK} where an explicit link is made between Bouchaud's trap model and biased random walk on random conductances.\\
The last possible scenario is when the conductances could be extremely small, which is what we are mainly interested in here. The extreme case would be percolation where the random walk is recurrent as soon as the percolation is subcritical. In our case, the conductances remain positive but have an heavy-tail at $0$. This creates ``barriers'' of edges with atypically small conductances that can make the walker come back to the root infinitely often, even when the tree is transient for simple random walk. Let us now describe our results. \\

Recall that $\mathcal{T}$ is an infinite, locally finite, tree and let $E$ be the set of all its edges. Let $(C_e)_{e\in E}$ be a collection of i.i.d.~random conductances that are almost surely positive. Moreover, assume  that 
\begin{equation}\label{distr-cond}
 {\bf P}\left(C_e \le \frac 1t\right) = \frac{L(t)}{t^m},  \qquad \mbox{for $t >0$},
\end{equation}
where $L:\mathbb{R}\to\mathbb{R}$ is a slowly-varying function. {For simplicity, we will also assume that ${\bf P}\left(C_e \ge 1\right)>0$ without loss of generality.}

For a realisation of the environment $(C_e)$, we can define a random walk on these conductances which jumps through an edge $e$ with a probability proportional to $C_e$.  
 For a formal definition of this random walk on random conductances (\Go), we refer to Section~\ref{RWRCs}.  In the following, we say that a walk is {\it transient} if it does not return to its starting point with positive probability. If a walk is not transient, it comes back to the root almost surely and it is called {\it recurrent}. {We also give a formal definition of recurrence and transience in Section~\ref{RWRCs}.}\\
{ Finally, the branching  -ruin number of $\Tcal$, formally defined in  \eqref{branchingdef2},  is  denoted by $br_r(\Tcal)$.}
\begin{theorem}\label{mainth}
Fix an infinite, locally finite, tree $\mathcal{T}$ and let $b =br_r(\mathcal{T})\in[0,\infty]$ be its branching-ruin number. If $b<1$, then {\Go } is recurrent. Assuming $b>1$, 
if $ m b >1$ then \Go$\,$   is transient and if $m b<1$ then it is recurrent.
\end{theorem}

\subsection{The $M$-digging random walk}
\label{subsection:introMdigging} 

Our second main result concerns a model of multi-excited random walks on trees, also known as {\it cookie random walks}.\\
Excited random walks were introduced by Benjamini and Wilson in~\cite{BW}  on $\mathbb Z^d$, and have been extensively studied (see \cite{ABK, BR, K03, K05,  V10}). Zerner \cite{Z05, Z06} introduced a generalization of this model called multi-excited random walks (or cookie random walk). These walks are well understood on $\mathbb{Z}$, but not much is known in higher dimensions.\\
Here, we study an extreme case of multi-excited random walks on trees, introduced by Volkov \cite{Volk}, called the $M$-digging random walk ($M$-DRW). We also study its biased version and generalize a result by Basdevant and Singh \cite{BasdSingh}, see Theorem \ref{th:digging0}, who studied it on regular trees.\\

Assign to each vertex $M$ cookies, where $M$ is a non-negative integer. Define a nearest-neighbor random walk ${\bf X}$ as follows. Each time it visits a vertex, if there is any cookie left there, it eats one of them and then jumps to the parent of that vertex. If no cookies are detected, then it jumps to one of the neighbors with uniform probability.  We refer to section~\ref{Mdig} for a formal definition {of this process.} \\

Volkov \cite{Volk} conjectured that this process is transient on any tree containing the binary, which was proved by Basdevant and Singh \cite{BasdSingh}. Here, we obtain a much finer description of the process and we can prove that this random walk actually undergoes a phase transition on trees with polynomial gowth, i.e.~on trees $\Tcal$ where the branching-ruin number $br_r(\Tcal)$ is finite.

\begin{theorem}
\label{th:digging}
Let $\mathcal{T}$ be an infinite, locally-finite, rooted tree, and let $M\in\mathbb{N}$. If  $br_r(\mathcal T)<{M+1}$ then $M$-DRW is recurrent and if $br_r(\mathcal T)>{M+1}$ then $M$-DRW is transient.
\end{theorem}

We refer to Theorem~\ref{th:digging0} for the more general result on the biased case and Theorem~\ref{maintheorem} for the case where the number of cookies on each vertex is inhomogeneous over the tree. 

\section{The models}

In this section, we define  relevant vocabulary and conventions. We then recall the definition of the {\it branching number} and {\it branching-ruin number} of a tree, and finally we formally define the models.

\subsection{Notation}
 Let  $\mathcal{T}=(V,E)$ be an infinite, locally finite, rooted tree with set of vertices $V$ and set of edges $E$. Let $\r$ be the root of $\mathcal{T}$.\\
 Two vertices $\nu,\mu\in V$ are called {\it neighbors}, denoted $\nu\sim\mu$, if $\{\nu,\mu\}\in E$.\\
For any vertex $\nu \in V\setminus\{{\r}\}$, denote  by $ {\nu}^{-1}$ its parent, i.e.  the neighbour of $\nu$ with shortest distance from $\r$.\\
For any $\nu \in V$, let  $|\nu|$ be  the number of edges in the unique self-avoiding path connecting $\nu$ to~$\r$ and call $|\nu|$ the {\it generation} of $\nu$. In particular, we have  $|\r|=0$.\\
For any edge $e \in E$
denote by  $e^-$ and $e^+$ its endpoints with $|e^+|=|e^-|+1$, and define the generation of an edge as $|e|=|e^+|$.\\
For any pair of  vertices $\nu$ and $\mu$, we write $\nu \le \mu$ if $\nu$ is on the unique self-avoiding path between ${\r}$ and $\mu$ (including it), and $\nu< \mu$ if moreover $\nu \neq \mu$. Similarly, {for two edges $e$ and $g$, we write $g\le e$ if $g^+ \le e^+ $ and $g<e$ if moreover $g^+\neq e^+$}. {For two vertices $\nu<\mu\in V$, we will denote by $[\nu,\mu]$  the unique self-avoiding path connecting $\nu$ to $\mu$.} For two neighboring vertices $\nu$ and $\mu$, we use the slight abuse of notation $[\nu,\mu]$ to denote the edge with endpoints $\nu$ and $\mu$ (note that we allow $\mu<\nu$).\\
 For two edges $e_1,e_2\in E$, we denote $e_1\wedge e_2$ the vertex with maximal distance from $\r$ such that $e_1\wedge e_2\le e_1^+$ and $e_1\wedge e_2\le e_2^+$.

\subsection{The Branching Number and The Branching-Ruin Number}
In order to define the branching number and the branching-ruin number of a tree, we will need the notion of {\it cutsets}.\\
Let $\mathcal T$ be an infinite, locally finite and rooted tree. A cutset in $\mathcal{T}$  is a set $\pi$ of edges such that, for any infinite self-avoiding path $(\nu_i)_{i\ge0}$ started at the root,  there exists a unique $i\ge 1$ such that $[\nu_{i-1},\nu_i]\in \pi$. In other words, a cutset is a minimal set of edges separating the root from infinity. We use $\Pi$ to denote the set of cutsets.\\

The branching number of $\mathcal{T}$ is defined as
\begin{equation}\label{branchingdef}
br(\mathcal T)   \Def  \sup \left\{  \gamma  >0 :  \inf_{\pi\in\Pi} \ \sum_{e\in
      \pi}\gamma^{-\left  |  e \right  |}>0  \right  \}\in[1,\infty].
\end{equation}
 branching-ruin number of $\mathcal{T}$ is defined as
\begin{equation}\label{branchingdef2}
br_r(\mathcal T)  \Def  \sup \left\{  \gamma  >0 :  \inf_{\pi\in\Pi} \ \sum_{e\in
      \pi}|e|^{-\gamma}>0  \right  \}\in[0,\infty].
\end{equation}

These quantities provide  good ways to measure respectively the exponential growth and the polynomial growth of a tree. For instance, a tree which is spherically symmetric (or regular) and whose $n$ generation grows like $b^n$, for $b\ge1$, has a branching number  equal to $b$. On the other hand, if such a tree grows like $n^b$, for some $b\ge0$, its branching-ruin number is equal to $b$. We refer the reader to \cite{LP} for a detailed investigation of the branching number and \cite{CKS} for discussions on the branching-ruin number.

\subsection{Formal definition of the models}
\label{definitionofmodel}

\subsubsection{The random walk on heavy-tailed random conductances}\label{RWRCs}
In this section, we provide a formal definition of the random walk on random conductances (\Go).\\
First let us define the environment of the walk. To the edges of $\Tcal$, we associate i.i.d.~random conductances $C_e\in(0,\infty)$, $e\in E$, with common law ${\bf P}$, where ${\bf E}$ denotes the corresponding expectation. We will assume that
\begin{equation}\label{distr-cond}
 {\bf P}\left(C_e \le \frac 1t\right) = \frac{L(t)}{t^m},  \qquad \mbox{for $t >0$},
\end{equation}
where $L:\mathbb{R}\to\mathbb{R}$ is a slowly varying function.

Given a realisation of the  environment $(C_e)_{e\in E}$, we define a reversible Markov chain ${\bf X}=(X_n)_n$. We denote $P^\omega_\nu$ the law of this Markov chain when it is started from a vertex $\nu\in V$. Under $P^\omega_\r$, we have that $X_0=\r$ and, if $X_n=\nu$ and $\mu\sim \nu$, we have that
\[
P^\omega_\r\left( \left. X_{n+1}=\mu \right|X_n=\nu \right)=P^\omega_\nu\left(X_{1}=\mu  \right)=\frac{C_{[\nu,\mu]}}{\sum_{\mu'\sim \nu}C_{[\nu,\mu']}  }. 
\]

We call $P^\omega_\cdot$ the {\it quenched law} of the random walk and denote $E^\omega_\cdot$ the corresponding expectation. We define the {\it annealed law} of ${\bf X}$ started at $\r$ as the semi-direct product $\P_\r={\bf P}\times P^\omega_\r$, that is the random walk averaged over the environment. We denote $\E_\r$ the corresponding annealed expectation.\\
For a vertex $v\in V$, $T(v)$ stands for the {\it return time} to $v$, that is
\[
T(v) \Def \inf\{n>0:X_n=v\}.
\]
A \Go~is said to be {\it recurrent} if  it  returns to ${\r}$,   $\mathbb{P}_\r$-almost surely. This process is {\it transient} if it is not recurrent, that is
\[
\P_\r\Big(T({\r}) =\infty\Big) >0.
\]
As $\P_\r\Big(T({\r}) =\infty\Big)={\bf E}\left(P^\omega_\r\Big(T({\r}) =\infty\Big)\right)$, ${\bf X}$ is transient if, with positive ${\bf P}$-probability, we have that
\[
P^\omega_\r\Big(T({\r}) =\infty\Big)>0.
\]
Finally, as ${\bf X}$ is a Markov chain under $P^\omega_\cdot$, we have that it is transient if and only if the walk {returns finitely often } to the root $\r$ and, using a zero-one law on the environment, we can prove that this happens with probability $0$ or $1$. Therefore, the notions of recurrence and transience are well defined in the quenched and annealed sense.

\subsubsection{The $M$-digging random walk}\label{Mdig}
Let $\mathcal{T}=(V,E)$ be an infinite, locally-finite, tree rooted at a vertex $\r$. We are going to define a {biased version of the $M$-DRW described above, which will also allow for an inhomogeneous initial number of cookies}.\\
Let $\overline{M}=(m_\nu,\nu\in V)$ be a collection of non-negative integers, with $m_\r=0$, and fix $\lambda>0$. For convenience, for $e\in E$, we denote $m_e=m_{e^+}$.\\
Let us define a random walk ${\bf X}=(X_n)_{n\geq 0}$ as follows. 
For any vertex $\nu\in V$, define
\begin{equation}
\ell_n(\nu)=\left|\left\{k\in\{0,\dots,n\}:\  X_{k}=\nu \right\}\right| .
\end{equation}
For each edge $e\in E$ and each time $n\in\mathbb{N}$, we associate the following weight:
\begin{equation}\label{weight}
W_n(e) \Def \left(1-\1_{\{\ell_n(e^-)\le m_{e^-}\}}\right)\lambda^{-|e|+1}.
\end{equation}
{As can be seen in \eqref{showandtell} below, the model remains unchanged if, in the above definition, we use $\lambda^{-|e|}$ instead of $\lambda^{-|e|+1}$. Our choice turns out to be  convenient in the proofs}.\\
For a non-oriented edge ${[\nu,\mu]}$, we will simply write $W_n(\nu,\mu)=W_n(\mu,\nu)=W_n({[\nu,\mu]})$  
We start the random walk at $X_0=\r$. At time $n\ge0$, for any $\nu\in V$, on the event $\{X_n=\nu\}$, we define, for any $\mu\sim \nu$,
\begin{equation}\label{showandtell}
\mathbb{P}\left(\left.X_{n+1}=\mu\right|\mathcal{F}_n\right)=\frac{W_n(\nu,\mu)}{\sum_{\mu'\sim\nu}W_n(\nu,\mu')},
\end{equation}
where $\mathcal{F}_n=\sigma(X_0,\dots,X_n)$ is the $\sigma$-field generated by the history of ${\bf X}$ up to time $n$. We call this walk an $M$-digging random walk with bias $\lambda$ and denote it $M$-DRW$_\lambda$.\\
It will be very convenient to observe ${\bf X}$ only at times when it is on vertices with no more cookies. For this purpose, let us define  $\widetilde{\bf X}=(\widetilde{X}_n)_n$ a nearest-neighbor random walk on $\mathcal{T}$ as follows.
Let $\sigma_0=0$ and, for any $n\in\mathbb{N}$,
\begin{equation}
\sigma_{n+1}=\inf\left\{k>\sigma_n: X_k\neq X_{\sigma_n}, \ell_k(X_k)\ge m_{X_k}+1\right\}.
\end{equation}
We define, for all $n\in\mathbb{N}$, $\widetilde{X}_n=X_{\sigma_n}$.\\
Next, we want to define notions of recurrence and transience for ${\bf X}$. As above, we define  the {\it return time} of ${\bf X}$, or $\widetilde{\bf X}$, to a vertex $\nu\in V$ by
\begin{equation}\label{defTnu}
T(\nu)\Def \inf\{k\ge1: \widetilde{X}_k=\nu\}.
\end{equation}
In words, we consider that a vertex $\nu$ is {\it hit} by ${\bf X}$ when it is hit by $\widetilde{\bf X}$ in the usual sense. The fact to choose this time to be greater than 1 will be convenient technically to accommodate with the particularities of the root.\\
We say that ${\bf X}$, or $\widetilde{\bf X}$, is {\it transient} if
\begin{equation}
\mathbb{P}\left(T(\r)=\infty\right)>0.
\end{equation}
Otherwise, we say that ${\bf X}$, or $\widetilde{\bf X}$, is {\it recurrent}.\\

Note that if we choose $m_\nu=M\in\mathbb{N}$ for all $\nu\in V\setminus\{\r\}$ and $\lambda=1$, then ${\bf X}$ is the $M$-DRW described in Section \ref{subsection:introMdigging}.

\section{Main results}

We are about to state a sharp criterion of recurrence/transience in terms of a quantity $RT(\mathcal{T},{\bf X})$, first introduced in \cite{CKS}.\\
For a function $\psi:E\to \mathbb{R}^+$, we define the quantity
\begin{equation}
\label{equ:RT}
RT(\mathcal{T},\psi)\Def \sup\left\{\gamma > 0: \inf_{\pi \in \Pi}\sum_{e\in \pi}\left(\prod_{g\le e}\psi(g)\right)^\gamma>0 \right\}.
\end{equation}

As we will see, for the relevant function $\psi$, the recurrence or transience of the walks will be related to this quantity being smaller or greater than $1$.

\subsection{Main results about \Go}\label{resultsrwrc}
It is straightforward to see that the two following results together imply  Theorem \ref{mainth}. {The proof of Proposition~\ref{mainth0}  is given in Section~\ref{sect-main}.}

Let us define, for any $e\in E$, $\psi_{{RC}}(e)=1$ if $|e|=1$ and, if $|e|>1$, 
\begin{equation}
\label{equ:definitionofpsiRC} 
\psi_{RC}(e)=\frac{\sum_{g<e}C_g^{-1}}{\sum_{g\le e}C_g^{-1}}.
\end{equation}
\begin{proposition}\label{mainth0}
Fix an infinite, locally finite, tree $\mathcal{T}$ and let $b =br_r(\mathcal{T})\in[0,\infty]$ be its branching-ruin number. {If $b< 1$  then $RT(\mathcal{T},\psi_{{RC}})<1$, ${\bf P}$-almost surely. Assuming $b >1$,}  we have that
\begin{enumerate}
 \item if $mb>1$ then $RT(\mathcal{T},\psi_{RC})>1$ with positive ${\bf P}$-probability;
 \item  if $mb<1$ then  $RT(\mathcal{T},\psi_{RC})<1$, ${\bf P}$-almost surely.
 \end{enumerate}
\end{proposition}

{The following result is  a direct consequence of Theorem 5 of \cite{CKS},  recalling the discussion at the end of Section~~\ref{RWRCs} and noting} that condition (2.5) in \cite{CKS} is trivially satisfied by Markov chains, which in that context is translated into  non-reinforced environments. Therefore, we will omit its proof.

\begin{proposition}[Theorem 5 of \cite{CKS}]
\label{prop:likeCKS}
Fix an infinite, locally finite, tree $\mathcal{T}$.  We have that
\begin{enumerate}
 \item if $RT(\mathcal{T},\psi_{RC})>1$ with positive ${\bf P}$-probability then {\Go } is transient;
 \item  if $RT(\mathcal{T},\psi_{RC})<1$ ${\bf P}$-almost surely then {\Go } is recurrent.
 \end{enumerate}
\end{proposition}

\subsection{Main results about the $M$-DRW$_\lambda$}\label{resultsdrw}

The following Theorem is more general than Theorem \ref{th:digging} in the introduction and deals with the homogeneous case where $\overline{M}=(m_\nu; \nu\in V)$ is such that $m_\r=0$ and $m_\nu=M$ for all $\nu\in V\setminus \{\r\}$. Let us emphasize that, in item $(1)$ {below},  the phase transition is given in terms of branching-ruin number whereas, in item $(2)$, the phase transition is given in terms of branching number.

\begin{theorem}
\label{th:digging0}
Let $\mathcal{T}$ be an infinite, locally-finite, rooted tree, and let $M\in\mathbb{N}$, $\lambda>0$. Denote ${\bf X}$ the  $M$-DRW$_\lambda$ on $\mathcal{T}$ with parameters $\lambda>0$ and  $\overline{M}=(m_\nu; \nu\in V)$ such that $m_\r=0$ and $m_\nu=M$ for all $\nu\in V\setminus \{\r\}$. We have that
\begin{enumerate}
\item in the case $\lambda=1$,   if  $br_r(\mathcal T)<{M+1}$ then ${\bf X}$ is recurrent and if $br_r(\mathcal T)>{M+1}$ then ${\bf X}$ is transient;
\item for any $\lambda>1$, if  $br(\mathcal T)<\lambda^{M+1}$ then ${\bf X}$ is recurrent and if $br(\mathcal T)>\lambda^{M+1}$ then ${\bf X}$ is transient;
\item for any $\lambda<1$,  ${\bf X}$ is transient.
\end{enumerate}
\end{theorem}

\begin{remark}
If, for a tree $\mathcal{T}$, $br(\mathcal{T})>1$, then we have that $br_r(\mathcal{T})=\infty$, as proved of Case V of the proof of Lemma \ref{mainlemma}. Therefore, the items $(1)$ and $(2)$ in Theorem \ref{th:digging0} are not contradictory.
\end{remark}
Note that, for a $b$-ary tree, $br(\mathcal T)=b$ and our result therefore agrees with Corollary 1.7 of \cite{BasdSingh}. In \cite{BasdSingh}, the authors prove that the walk is recurrent at criticality on regular trees, but this is not expected to be true in general.\\

We are about to state a sharp criterion of recurrence/transience in terms of a quantity $RT(\mathcal{T},\cdot)$ as defined in \eqref{equ:RT}, which will apply to the general case $\overline{M}=(m_\nu; \nu\in V)\in\mathbb{N}^V$. We will then prove that Theorem \ref{th:digging0} is a simple corollary of this general result.\\

For this purpose, we need some notation.
Let us define a function $\psi_{M,\lambda}$ on the edges of $E$ such that, for any $e\in E$,  $\psi_{M,\lambda}(e)=1$ if $|e|=1$ and, for any $e\in E$ with $|e|>1$,
\begin{equation}
\label{psi}
\begin{split}
    \psi_{M,\lambda}(e)&\Def\left (\frac{\lambda^{|e|-1}-1}{\lambda^{|e|}-1} \right)^{m_{e^+}+1}\text{ if }\lambda\neq1,\\
        \psi_{M,\lambda}(e)&\Def \left (\frac{|e|-1}{|e|} \right)^{m_{e^+}+1}\text{ if }\lambda=1.
    \end{split}
\end{equation}

As we will see in Section \ref{rubin}, $\psi_{M,\lambda}(e)$ corresponds to the probability that ${\bf X}$, or $\widetilde{\bf X}$, when restricted to $[\r,e^+]$ (i.e.~the path from the root to $e^+$), hits $e^+$ before returning to $\r$, after  having hit $e^-$. \\
We will prove the following result in Section \ref{sect_thdigg}.
\begin{theorem}
  \label{maintheorem}
  Consider an $M$-DRW$_\lambda$ ${\bf X}$ on an infinite, locally finite, rooted tree $\mathcal T$, with parameters $\lambda>0$ and $\overline{M}=(m_\nu; \nu\in V)\in\mathbb{N}^V$. If $RT(\mathcal{T},\psi_{M,\lambda})<1$ then ${\bf X}$ is recurrent. If $RT(\mathcal{T},\psi_{M,\lambda})>1$ and if
  \begin{equation}\label{condition1}
  \exists M\in \mathbb{N} \text{ such that } \sup_{\nu\in V}m_\nu \leq M,
\end{equation}
then ${\bf X}$ is transient. 
\end{theorem}

{The following result concerns the homogeneous case.  Theorem \ref{th:digging0}  is  a straightforward consequence of Theorem \ref{maintheorem} and Lemma \ref{mainlemma}.}

\begin{lemma}\label{mainlemma}
Consider an $M$-DRW$_\lambda$ ${\bf X}$ on an infinite, locally finite, rooted tree $\mathcal T$, with parameters $\lambda>0$ and $M=(m_\nu; \nu\in V)$ such that $m_\r=0$ and $m_\nu=M$ for all $\nu\in V\setminus \{\r\}$. We have that
\begin{enumerate}
\item for $\lambda=1$, if  $br_r(\mathcal T)<{M+1}$ then $RT(\mathcal{T},\psi_{M,\lambda})<1$ and if $br_r(\mathcal T)>{M+1}$ then $RT(\mathcal{T},\psi_{M,\lambda})>1$;
\item for $\lambda>1$,  if  $br(\mathcal T)<\lambda^{M+1}$ then $RT(\mathcal{T},\psi_{M,\lambda})<1$ and if $br(\mathcal T)>\lambda^{M+1}$ then $RT(\mathcal{T},\psi_{M,\lambda})>1$;
\item for $\lambda<1$, we have $RT(\mathcal{T},\psi_{M,\lambda})<1$.
\end{enumerate}
\end{lemma}
The proofs of Theorem~\ref{maintheorem} and Lemma~\ref{mainlemma} are given in Section~\ref{sec-digg}.
 

\section{Preliminary results} \label{randomtree}

Proposition \ref{propperc} below can be proved following line by line the argument in Section 8 of \cite{CKS}. For the sake of completeness, we give an outline of the proof in the Appendix \ref{appendix}. It relies on the concept of quasi-independent percolation defined as below (see also \cite{LP}, page 144).  In the following, we denote {by} $\mathcal{C}(\r)$ the cluster of open edges containing the root $\r$.

\begin{definition}\label{defquasi}
An edge-percolation is said to be  \emph{quasi-independent} if there exists a constant $C_Q\in(0,\infty)$ such that, for any two edges $e_1,e_2\in E$ with common ancestor $e_1\wedge e_2$, we have that
\begin{equation}\label{qindep0} 
\begin{split}
\bP\big(\left.e_1,e_2\in\C(\r)\right|e_1\wedge e_2\in \C(\r)\big)\le& C_Q\bP\big(\left.e_1\in\C(\r)\right|e_1\wedge e_2\in \C(\r)\big)\\
&\times \bP\big(\left.e_2\in\C(\r)\right|e_1\wedge e_2\in \C(\r)\big).
\end{split}
\end{equation}
\end{definition}
This previous notion is useful when one tries to prove the super-criticality of a correlated percolation.

\begin{proposition}\label{propperc}
Consider an edge-percolation ({not necessarily} independent), such that edges at generation 1 are open almost surely and, for $e_1\in E$ with $|e_1|>1$,
\begin{equation}\label{gendefperco}
\mathbf{P}\left(\left.e_1\in\mathcal{C}(\r)\right|e_0\in\mathcal{C}(\r)\right)=\psi(e_1)>0,
\end{equation}
where $e_0\sim e_1$ and $e_0<e_1$. If $RT(\mathcal{T},\psi)<1$ then $\mathcal{C}(\r)$ is finite almost surely. If the percolation is quasi-independent and if $RT(\mathcal{T},\psi)>1$ then $\mathcal{C}(\r)$ is infinite with positive probability.
\end{proposition}

The proof of Proposition \ref{propperc} above is postponed in Appendix \ref{appendix}.\\

Let us first apply this to a particular percolation in order to obtain a sufficient criterion for subcriticality.

\begin{corollary}\label{coroperc}
Let $\mathcal{T}$ be a tree with branching ruin number $br_r(\mathcal{T})=b\in[0,\infty]$. Fix a parameter $\delta >0$ and perform a  percolation (not necessarily independent) on $\mathcal{T}$ such that \eqref{gendefperco} holds and assume moreover that $\psi(e)= 1-\delta |e|^{-1}$ as soon as $|e|>n_0$, for some {integer} $n_0>1$. If $\delta>b$ then the percolation is  subcritical.
\end{corollary}

\begin{proof}
For a cutset $\pi$, let $|\pi| = \inf\{|e| \colon e \in \pi\}$.
First, note that for any $\alpha>b$,  {
\[
\inf_{\pi\in\Pi:|\pi|\le n_0}\sum_{e\in\pi} |e|^{-\alpha}\ge n_0^{-\alpha}>0,
\]
and therefore}
\[
\inf_{\pi\in\Pi:|\pi|>n_0}\sum_{e\in\pi} |e|^{-\alpha}=\inf_{\pi\in\Pi}\sum_{e\in\pi} |e|^{-\alpha}=0.
\]
Second, for any $\gamma>b/\delta$, we have 
\begin{equation}
\begin{split}
\inf_{\pi\in\Pi}\sum_{e\in\pi} \prod_{g\le e}\left(\psi(g) \right)^\gamma &\le \inf_{\pi\in\Pi:|\pi|>n_0}\sum_{e\in\pi} \prod_{g\le e}\left(\psi(g) \right)^\gamma\\
&\le\inf_{\pi\in\Pi:|\pi|>n_0}\sum_{e\in\pi} \prod_{g\le e} \left(1-\delta |g|^{-1}\right)^{\gamma}\\
&\le\inf_{\pi\in\Pi:|\pi|>n_0}\sum_{e\in\pi}  \exp\left(-\gamma\delta\sum_{i=1}^{|e|} i^{-1}\right)\\
&\leq \inf_{\pi\in\Pi:|\pi|>n_0}\sum_{e\in\pi} |e|^{-\gamma\delta}=0.
\end{split}
\end{equation}
Hence $RT(\mathcal{T},\psi)<1$ and by using Proposition~\ref{propperc} the cluster $\mathcal{T}_\delta$ is finite, almost surely.
\end{proof}

{Next,} we use Proposition~\ref{propperc} and Corollary \ref{coroperc} to prove the following result.

\begin{proposition} \label{siham}
Let $\mathcal{T}$ be a tree with branching ruin number $br_r(\mathcal{T})=b\in[0,\infty]$. Fix a parameter $\delta >0$ and perform a quasi-independent percolation on $\mathcal{T}$ such that \eqref{gendefperco} holds and assume moreover that $\psi(e)\ge 1-\delta |e|^{-1}$ as soon as $|e|>n_0$, for some {integer} $n_0>1$. Let $\mathcal{T}_\delta$ be the connected cluster containing the root $\r$. We have that
\begin{enumerate}
\item if $\delta <b$ then  $\mathcal{T}_\delta$  is infinite with positive probability;
\item {for any  $\delta\in(0,b)$ we have that, with positive probability, $br_r(\mathcal{T}_\delta)\ge b-2\delta$.}
\end{enumerate}
\end{proposition}

\begin{proof} 
First we prove $(1)$.
{For $\pi\in \Pi$, we define} $|\pi|=\min\{|e|; e\in \pi\}$. {Notice}  that, for any $\gamma>1$, as $\psi(e)>0$ for every $e\in E$,
\begin{equation}\label{lonelyview}
\inf_{\pi\in\Pi:|\pi|\le n_0}\sum_{e\in\pi} \prod_{g\le e}\left(\psi(g) \right)^\gamma>0.
\end{equation}

If $\delta<b$, then for any $\gamma\in(1,b/\delta)$, we have 
\begin{equation}\label{lonelyview2}
\begin{split}
\inf_{\pi\in\Pi:|\pi|>n_0}\sum_{e\in\pi} \prod_{g\le e}\left(\psi(g) \right)^\gamma
&\ge \inf_{\pi\in\Pi:|\pi|>n_0}\sum_{e\in\pi} \prod_{g\le e} \left(1-\delta |g|^{-1}\right)^{\gamma}\\
&\ge c\inf_{\pi\in\Pi}\sum_{e\in\pi}  \exp\left(-\gamma\delta\sum_{i=1}^{|e|} i^{-1}\right)\\
&\ge 2^{-b}c \inf_{\pi\in\Pi}\sum_{e\in\pi} |e|^{-\gamma\delta}>0,
\end{split}
\end{equation}
where $c$ is some positive constant. Putting \eqref{lonelyview} and \eqref{lonelyview2} together, we have that $RT(\mathcal{T},\psi)>1$. By Proposition \ref{propperc}, as the percolation is quasi-independent, the cluster $\mathcal{T}_\delta$ is infinite with positive probability.\\
Next, we turn to the proof of $(2)$.
Consider the previous percolation, with $\delta<b$ and  fix $p<b-\delta$.\\
On the event $\{\mathcal{T}_\delta $ is infinite$\}$, which has positive probability, we perform an independent percolation  on $\mathcal{T}_\delta$ for which an edge $e$ stays open with probability $(1-p|e|^{-1})$. We proved that if  $p<br_r(\mathcal{T}_\delta)$ then the percolation is supercritical and if $p>br_r(\mathcal{T}_\delta)$ then it  is subcritical. We denote $\mathcal{T}_{\delta+p}\rq$ the resulting cluster of the root.\\
On the other hand, performing this percolation on $\mathcal{T}_\delta$ is equivalent to performing a quasi-independent percolation on the whole tree $\mathcal{T}$ where an edge $e$ stays open with probability $\psi(e)(1-p|e|^{-1})$.
As $\psi(e)(1-p|e|^{-1})\ge (1-\delta |e|^{-1})(1-p|e|^{-1})\ge 1-(\delta+p)|e|^{-1}$, for $|e|>n_0$, if $p+\delta<b$, this percolation is supercritical, i.e.~$\mathcal{T}_{p+\delta}\rq$ is infinite with positive probability.\\
This implies that, on the event $\{\mathcal{T}_\delta $ is infinite$\}$, the cluster $\mathcal{T}_{\delta+p}\rq$ is infinite with positive probability. Therefore, by Corollary~\ref{coroperc}, $br_r(\mathcal{T}_\delta)\ge p$ with positive probability. As this holds for any $p<b-\delta$, we obtain the conclusion.
\end{proof}

%
%

\section{Proof {of Proposition~\ref{mainth0} and} Theorem  \ref{mainth}}\label{sect-main}

First, note that  Theorem  \ref{mainth} is a straightforward consequence of Proposition \ref{mainth0} and Proposition~\ref{prop:likeCKS}. Therefore, it remains to prove  Proposition \ref{mainth0}.

\subsection{Transience: proof of the first item of Proposition \ref{mainth0}} \label{sec:transience}
In this section,  we will prove that $RT(\mathcal{T},\psi_{RC})>1$, where we recall that this quantity is defined in \eqref{equ:RT} and $\psi_{RC}$ is defined in \eqref{equ:definitionofpsiRC}.\\
In particular, we can rewrite
\begin{align}\label{defBR}
RT(\mathcal{T},\psi_{RC}) = \sup\left\{\lambda>0: \inf_{\pi \in \Pi } \sum_{e \in \pi} \left(\frac{1}{\sum_{i \le e} C^{-1}_i}\right)^{\lambda}>0\right\}. 
\end{align}
Besides, notice that $\psi(e)$ represents the probability that a one-dimensional random walk on the conductances $(C_e)_{e\in E}$, restricted to the ray connecting $\r$ to $e^+$ and  started at  $e^-$, hits $e^+$  before {returning to}  $\r$. 

\begin{proposition}\label{conductbound}  For any $p \in \N$, and for {any $\tau>0$}, there exists a positive finite constant $K_{p,\tau}$ such that
\begin{equation}\label{cond1}
{\bf E}\left[\Big(\sum_{i=1}^n C^{-1}_i\Big)^p\;\Big{|}\; \bigcap_{i=1}^n \{C^{-1}_i \le  i^{\frac{1+\tau}{m}}\}\right] \le K_{p,\tau} n^{p (1 \vee\frac {(1 + \tau)^2}m) } , \qquad \mbox{for all } n \in \N.
\end{equation} 
\end{proposition}

\begin{proof}
{Recall that for any non-negative random variable $Z$ we have, for $a >1$, 
\[
{\bf E}[Z^a] = \int_0^\infty a u^{a-1} \P(Z \ge u) \d u.
\]
For any $b>0$ we have that any slowly varying function $L(u)$ is $o(u^{b})$, as $u \to \infty$. Hence, for any $\tau>0$, there exists a constant $K_\tau,i_0>0$ depending only on $L$ and $\tau$, such that, for $i\ge i_0$,
\begin{equation}\label{cond2}
\begin{aligned}
{\bf E}[C_i^{-a}\;|\; C_i^{-1} \le i^{\frac{1+\tau}{m}}] &\le\left(1+ \int_{1}^{ i^{\frac{1+\tau}{m}}} a u^{a-1} \frac{L(u)}{u^m}\d u \right)\left( \frac 1{1- i^{-({1+\tau})} L( i^{\frac{1+\tau}{m}})}\right)\\
&\le 2\left(1+\frac{K_\tau}{a-m} {i^{a(1+ \tau)^2/m -1 } -1}\right)\\
&\Def b^{\ssup {a, \tau}}_i.
\end{aligned}
\end{equation}
For simplicity we drop $\tau$ from the notation, and use $(b^{\ssup a}_i)_i$.   Notice that the sequence $(b^{\ssup a}_i)_i$, when $a \ge 1$,  is  $O(i^{\frac {a(1+\tau)^2}m -1}\vee 1)$, that is there exists $\widetilde{K}_a>0$ depending only on $L$, $a$ and $\tau$ such that 
$$b^{\ssup a}_i \le \widetilde{K}_a \Big(i^{\frac {a(1+ \tau)^2}m -1} \vee 1\Big),$$ 
for all $i \in \N$.  In order to prove the proposition,  we proceed by double induction.    First we prove that \eqref{cond1} holds for $p=1$ and all $n \in \N$. In fact, for $ m >0$, we have
\begin{equation}\label{cond1.5}
{\bf E}\left[\Big(\sum_{i=1}^n C^{-1}_i\Big)\;\Big{|}\; \bigcap_{i=1}^n \{C^{-1}_i \le i^{\frac{1+\tau}m}\}\right]   \le \sum_{i =1} ^n \widetilde{K}_1 (i^{\frac {(1+\tau)^2}m -1}\vee 1)= O(n^{(\frac {(1+ \tau)^2}{m}\vee 1)}).
\end{equation}
{Note that, in the previous inequality, we use that ${\bf P}[C_e\ge1]>0$ for any $e\in E$, so that the conditional probability on the left-hand side is well-defined.}\\
Assume that  \eqref{cond1} holds for all $p \le \beta -1$ and for all $n \in \N$.
Notice that \eqref{cond1}  is trivially true for $n=1$ and $ p = \beta$. Suppose it is true for all $n \le N$ and for $p = \beta$.   To simplify the notation, set $\eta = \frac{(1+\tau)^2}m \vee 1$. Next we prove the result for $N+1$.   We can suppose that $K_\beta$ is larger than 
\begin{equation}\label{cond1.6} \beta \max_{0 \le j \le \beta-1} {\beta \choose j} K_j \widetilde{K}_{\beta - j},
\end{equation}
where $K_0 = 1$.  We have 
\begin{equation}\label{cond2.1}
\begin{aligned}
&{\bf E}\left[\Big(\sum_{i=1}^{N+1} C^{-1}_i\Big)^\beta\;\Big{|}\; \bigcap_{i=1}^{N+1}\big \{C^{-1}_i \le i^{\frac{1+\tau}{m}}\big\}\right] \\
&={\bf E}\left[\Big(\sum_{i=1}^{N} C^{-1}_i\Big)^\beta + C^{-\beta}_{N+1}  +\sum_{j=1}^{\beta -1} {\beta \choose j} \Big(\sum_{i=1}^{N} C^{-1}_i\Big)^j C^{-\beta +j}_{N+1} \;\Big{|}\; \bigcap_{i=1}^{N+1} \{C^{-1}_i \le i^{\frac{1+\tau}{m}}\}\right]\\
&\le K_\beta N^{\beta \eta} + b^{\ssup \beta}_{N+1}  + \sum_{j=1}^{\beta -1} {\beta \choose j}\bE\left[\Big(\sum_{i=1}^{N} C^{-1}_i\Big)^j\;\Big{|}\; \bigcap_{i=1}^{N+1} \{C^{-1}_i \le i^{\frac{1+\tau}{m}}\}\right] b_{N+1}^{\ssup{\beta-j}}\\
&\le K_\beta N^{\beta\eta}  + \widetilde{K}_\beta \Big((N+1)^{\frac {\beta(1 + \tau)^2} m  -1}  \vee 1\Big)  + \sum_{j=1}^{\beta -1} {\beta \choose j} K_{j} N^{j\eta} \widetilde{K}_{\beta - j}  \Big((N+1)^{\frac{(\beta -j)(1+\tau)^2}m -1}  \vee 1\Big).
\end{aligned}
\end{equation}}

{In the step before the last one, we used independence between $C_{N+1}$ and $(C_i)_{ i \le N}.$   As we can choose $K_\beta$ to be larger than \eqref{cond1.6},  we have
\begin{equation}\label{cond3}
{\bf E}\left[\Big(\sum_{i=1}^{N+1} C^{-1}_i\Big)^\beta\;\Big{|}\; \bigcap_{i=1}^{N+1} \{C^{-1}_i \le i^{\frac{1+\tau}{m}}\}\right] \le K_\beta \left( N^{\beta \eta} +  (N+1)^{\beta \eta -1}\right).
 \end{equation}
It remains to prove that the right-hand side of \eqref{cond3} is less than $K_\beta (N+1)^{\beta \eta}.$ 
Notice that the right-hand side   of \eqref{cond3} equals 
$$ (N+1)^{\beta \eta}  K_\beta \left( \Big(1 - \frac 1{N+1}\Big)^{\beta \eta} + \frac 1{ N+1}\right)\le K_\beta (N+1)^{\beta \eta},
$$
where we used $(1- x)^a \le 1 - x$ for all $x \in (0,1)$ and $a >1$.}
\end{proof}

\begin{corollary}\label{coropercobound}
For any $\eps\in(0,1)$, any $t>0$, there exist $C_{\eps,t}>0$ such that, for any $e\in E$, we have that 
\[
{\bf P}\left(\left.\sum_{g\le e} C^{-1}_g>|e|^{\left(1\vee \frac{1}{m}\right)+\frac{m+3}{m}\eps}\right|\bigcap_{g\le e}\left\{ C^{-1}_g \le|g|^{\frac{1+\eps}{m}}\right\}\right)\le C_{\eps,t} |e|^{-t}.
\]
\end{corollary}

\begin{proof}
Using Proposition \ref{conductbound} and Markov's inequality gives that, for any $p\in\mathbb{N}$,
\begin{equation}
\begin{split}
{\bf P}\left(\left.\sum_{g\le e} C^{-1}_g>|e|^{\left(1\vee \frac{(1+\eps)^2}{m}\right)+\eps}\right|\bigcap_{g\le e}\left\{ C^{-1}_g \le|g|^{\frac{1+\eps}{m}}\right\}\right)\le K_{p,\eps} |e|^{-p\eps}.
\end{split}
\end{equation}
This gives the conclusion by choosing $p=\lceil t/\eps\rceil$ and by noting that $\left(1\vee \frac{(1+\eps)^2}{m}\right)+\eps\le \left(1\vee \frac{1}{m}\right)+\frac{m+3}{m}\eps$
\end{proof}

Next, we will define a quasi-independent percolation on the tree $\mathcal{T}$.
Let us fix $\eps\in(0,1\wedge b)$ small enough, such that the following conditions are satisfied
\begin{align}\label{epsm} 
(1+\eps)\frac{1+(m+3)\eps}{m}\le b-2\eps &\qquad \text{ if }bm>1,\\ \label{epsm2}
(1+4\eps)(1+\eps)\le b-2\eps &\qquad \text{ if } b>1.
\end{align}
Let us define the percolation such that, for $e\in E$ with $|e|=1$, $e$ is open almost surely and if $|e|>1$ then
\begin{equation}\label{defpercocond}
\left\{e\text{ is open}\right\}\Def \left\{C^{-1}_e \le |e|^{\frac{1+\eps}{m} } \right\}\cap\left\{\sum_{g\le e} C^{-1}_g\le|e|^{\left(1\vee \frac{1}{m}\right)+\frac{m+3}{m}\eps} \right\}.
\end{equation}
We will denote  by $\mathcal{T}_C$ the cluster of open edges containing the root.
Let us define the function $\psi_{C}$ on edges such that $\psi_C(e)=1$ if $|e|=1$ and, if $|e|>1$ and $e_0$ is the parent of $e$,  {that is the unique edge such that  $e_0^+ = e^-$},  then 
\begin{equation}
\psi_C(e)\Def{\bf P}\left(\left.e\in\mathcal{T}_C\right|e_0\in\mathcal{T}_C\right).
\end{equation}

\begin{proposition}\label{boundbrr}
The percolation defined by \eqref{defpercocond} is quasi-independent. Moreover, $RT(\mathcal{T},\psi_C)>1$ and, with positive ${\bf P}$-probability $br_r(\mathcal{T}_C)\ge b-\eps$.
\end{proposition}

\begin{proof}
Let us prove that there exists a constant $p_0>0$ such that, for any $e\in E$,
\begin{equation}\begin{split}\label{previousbound}
{\bf P}\left(e\in\mathcal{T}_C\Big|\bigcap_{g\le e} \left\{C^{-1}_g \le |g|^{\frac{1+\eps}{m} } \right\}\right)&={\bf P}\left(\left.\bigcap_{g\le e}\left\{g\in\mathcal{T}_C\right\}\right|\bigcap_{g\le e} \left\{C^{-1}_g \le |g|^{\frac{1+\eps}{m} } \right\}\right)\\
&\ge p_0.
\end{split}
\end{equation}
{Indeed, the conditioning in the above expression is equivalent to picking a sequence of independent conductances $(C_j)_{j\ge1}$ under a measure $\widetilde{\bf P}$ such that $C_j$ is picked under the conditioned law ${\bf P}(\cdot|C_j^{-1} \le j^{\frac{1+\eps}{m}})$, and looking at the events corresponding to the second event on the right hand side of \eqref{defpercocond}, that is
\[
A_j=\left\{\sum_{i\le j} C^{-1}_i\le j^{\left(1\vee \frac{1}{m}\right)+\frac{m+3}{m}\eps}\right\}.
\]
By Corollary \ref{coropercobound} (applied with $t=2$ for instance)  and Borel-Cantelli Lemma, there exists $k\in\mathbb{N}$ (deterministic) such that $\widetilde{\bf P}\left(\cap_{n\ge k} A_n\right)>0$. Now, if one replaces $C_j$ by $\tilde{C}_j=\max(C_j,1)$ for $1\le j \le k$, and let $\tilde{A}_n$ be the the same event as $A_n$ but where $C_j$ is replaced by $\tilde{C}_j$, then $\tilde{A}_1,\dots,\tilde{A}_k$ always happen and $\widetilde{\bf P}\left(\cap_{n\ge 1} \tilde{A}_n\right)\ge \widetilde{\bf P}\left(\cap_{n\ge k} A_n\right)>0$. Finally, we can choose
\[
p_0=\widetilde{\bf P}\left(\cap_{n\ge 1} A_n\right)=\widetilde{\bf P}\left(\cap_{n\ge 1} \tilde{A}_n\right)\times \widetilde{\bf P}\left(\cap_{1\le j\le k} \left\{C_j\ge 1\right\}\right)>0,
\]
which proves the claim \eqref{previousbound}.
 }

Let us prove that the percolation is quasi-independent. Let $e_1,e_2\in E$ and let $e$ be their common ancestor with highest generation. We have that
\begin{equation}
\begin{split}
&{\bf P}\left(e_1,e_2\in\mathcal{T}_C\Big|e\in\mathcal{T}_C\right)=\frac{\bP\left(e_1,e_2\in\mathcal{T}_C\right)}{\bP\left(e\in\mathcal{T}_C\right)}\\
=&\prod_{e<g \le e_1\text{ or }e<g \le e_2}\bP\left( C^{-1}_g \le |g|^{\frac{1+\eps}{m} } \right)\frac{\bP\left(e_1,e_2\in\mathcal{T}_C\Big|\bigcap_{g\le e_1,e_2} \left\{C^{-1}_g \le |g|^{\frac{1+\eps}{m} } \right\}\right)}{\bP\left(e\in\mathcal{T}_C\Big|\bigcap_{g\le e} \left\{C^{-1}_g \le |g|^{\frac{1+\eps}{m} } \right\}\right)}\\
\le&\frac{1}{p_0}\times\prod_{e<g \le e_1\text{ or }e<g \le e_2}\bP\left(C^{-1}_g \le |g|^{\frac{1+\eps}{m} } \right)\\
{=}&\frac{1}{p_0}\times\frac{\prod_{g \le e_1}\bP\left(C^{-1}_g \le |g|^{\frac{1+\eps}{m} }\right)}{\prod_{g \le e}\bP\left(C^{-1}_g \le |g|^{\frac{1+\eps}{m} }\right)}\times\frac{\prod_{g \le e_2}\bP\left(C^{-1}_g \le |g|^{\frac{1+\eps}{m} }\right)}{\prod_{g \le e}\bP\left(C^{-1}_g \le |g|^{\frac{1+\eps}{m} }\right)} \\
\le&\frac{1}{p_0^3}\times\frac{\prod_{g \le e_1}\bP\left(C^{-1}_g \le |g|^{\frac{1+\eps}{m} }\right)}{\prod_{g \le e}\bP\left(C^{-1}_g \le |g|^{\frac{1+\eps}{m} }\right)}\times\frac{\prod_{g \le e_2}\bP\left(C^{-1}_g \le |g|^{\frac{1+\eps}{m} }\right)}{\prod_{g \le e}\bP\left(C^{-1}_g \le |g|^{\frac{1+\eps}{m} }\right)} \\
&\times\frac{\bP\left(e_1\in\mathcal{T}_C\Big|\bigcap_{g\le e_1} \left\{C^{-1}_g \le |g|^{\frac{1+\eps}{m} } \right\}\right)}{\bP\left(e\in\mathcal{T}_C\Big|\bigcap_{g\le e} \left\{C^{-1}_g \le |g|^{\frac{1+\eps}{m} } \right\}\right)^2} \bP\left(e_2\in\mathcal{T}_C\Big|\bigcap_{g\le e_2} \left\{C^{-1}_g \le |g|^{\frac{1+\eps}{m} } \right\}\right)\\
{=}& \frac{1}{p_0^3}\bP\left(e_1\in\mathcal{T}_C\Big|e\in\mathcal{T}_C\right)\times \bP\left(e_2\in\mathcal{T}_C\Big|e\in\mathcal{T}_C\right),
\end{split}
\end{equation}
{where the first equality simply uses the definition of conditional probability, the second uses \eqref{previousbound} and bounds the probability in the numerator by $1$, the third is a simple re-writing, the fourth uses again \eqref{previousbound} and bounds the probability in the denominator by 1 and, finally, the fifth one is just using the definition of conditional probability.}\\
This proves that the percolation is quasi-independent.\\
{Let $e$ be a generic edge with $|e|>1$, and denote by $e_0$  its parent.}
Using \eqref{previousbound}, \eqref{defpercocond} and again  Corollary \ref{coropercobound}, we have that, there exists $c_0>0$ such that
\begin{equation}
\begin{split}
&{\bf P}\left(e\notin\mathcal{T}_C\Big|C^{-1}_e \le |e|^{\frac{1+\eps}{m} } , e_0\in\mathcal{T}_C\right)=\frac{\bP\left(e\notin\mathcal{T}_C,C^{-1}_e \le |e|^{\frac{1+\eps}{m} } , e_0\in\mathcal{T}_C\right)}{\bP\left(C^{-1}_e \le |e|^{\frac{1+\eps}{m} } , e_0\in\mathcal{T}_C\right)} \\
=&\frac{\bP\left(e\notin\mathcal{T}_C,C^{-1}_e \le |e|^{\frac{1+\eps}{m} } , e_0\in\mathcal{T}_C\right)}{\bP\left(e_0\in\mathcal{T}_C\right)\bP\left(C^{-1}_e \le |e|^{\frac{1+\eps}{m} }\right)} \\
=&\frac{\bP\left(e\notin\mathcal{T}_C,C^{-1}_e \le |e|^{\frac{1+\eps}{m} } , e_0\in\mathcal{T}_C\right)}{\bP\left(e_0\in\mathcal{T}_C\right)\bP\left(\bigcap_{g\le e} \left\{C^{-1}_g \le |g|^{\frac{1+\eps}{m} } \right\}\right)}  \bP\left(\bigcap_{g\le e_0} \left\{C^{-1}_g \le |g|^{\frac{1+\eps}{m} } \right\}\right) \\
\le&\frac{\bP\left(e\notin\mathcal{T}_C,C^{-1}_e \le |e|^{\frac{1+\eps}{m} },\bigcap_{g\le e_0} \left\{C^{-1}_g \le |g|^{\frac{1+\eps}{m} } \right\}\right)}{\bP\left(\bigcap_{g\le e} \left\{C^{-1}_g \le |g|^{\frac{1+\eps}{m} } \right\}\right)} \frac{\bP\left(\bigcap_{g\le e_0} \left\{C^{-1}_g \le |g|^{\frac{1+\eps}{m} } \right\}\right)}{\bP\left(e_0\in\mathcal{T}_C\right)} \\
\le& \frac{\bP\left(e\notin\mathcal{T}_C\Big|\bigcap_{g\le e} \left\{C^{-1}_g \le |g|^{\frac{1+\eps}{m} } \right\}\right)}{\bP\left(e_0\in\mathcal{T}_C\Big|\bigcap_{g\le e_0} \left\{C^{-1}_g \le |g|^{\frac{1+\eps}{m} } \right\}\right)}
\le \frac{c_0}{|e|^{{1+\eps}}}.
\end{split}
\end{equation}
Thus, we obtain that 
\begin{equation}
\begin{split}
1-\psi_C(e)&=\bP\left(\left.e\notin\mathcal{T}_C\right|e_0\in\mathcal{T}_C\right)\\
&\le \bP\left(C^{-1}_e > |e|^{\frac{1+\eps}{m} }\right)+\bP\left(e\notin\mathcal{T}_C\Big|C^{-1}_e \le |e|^{\frac{1+\eps}{m} } , e_0\in\mathcal{T}_C\right)\\
&\le \frac{c_0+L(|e|^{\frac{1+\eps}{m} })}{|e|^{1+\eps}}.
\end{split}
\end{equation}
Therefore, there exists $n_0>1$ such that, for any $e\in E$ with $|e|>n_0$, we have that
\[
\psi_C(e)\ge 1-\frac{\epsilon}{2}|e|^{-1}.
\]
By Proposition \ref{siham}, as the percolation defined by \eqref{defpercocond} is quasi-independent and $\eps<b$, we have that $br_r(\mathcal{T}_C)\ge b-\eps$ with positive probability.
\end{proof}
%
%
%

Let us consider different cases and prove that $RT(\mathcal{T}, \psi_{RC})>1$, where we refer to \eqref{defBR} for a definition of this quantity.

\begin{proposition}
If $m\in(0,1)$ and $bm>1$ then $RT(\mathcal{T},\psi_{RC})>1$ with positive ${\bf P}$-probability.
\end{proposition}

\begin{proof}
Recall the percolation $\mathcal{T}_C$ defined in \eqref{defpercocond}. {Let us denote $\Pi_C$ the set of all the cutsets in $\mathcal{T}_C$. By Proposition \ref{boundbrr}, we have that $br_r(\mathcal{T}_C)\ge b-\eps$ with positive ${\bf P}$-probability.}  On this event, we have that
\begin{equation}
\begin{split}
  \inf_{\pi \in \Pi } \sum_{e \in \pi} \left(\frac{1}{\sum_{i \le e} C^{-1}_i}\right)^{1+\eps}  &\ge   \inf_{\pi \in \Pi_C } \sum_{e \in \pi} \left(\frac{1}{\sum_{g \le e} C^{-1}_g}\right)^{1+\eps} \\
&  \ge\inf_{\pi \in \Pi_C } \sum_{e \in \pi} \left( |{e}|^{- \frac{1}{m}-\frac{m+3}{m}\eps}\right)^{1+\eps}\\
&  \ge\inf_{\pi \in \Pi_C } \sum_{e \in \pi}  |{e}|^{- (b-2\eps)}>0,
\end{split}
\end{equation}
where we used \eqref{epsm}. This implies that $RT(\mathcal{T}, \psi_{RC})>1$  with positive ${\bf P}$-probability, as defined in \eqref{defBR}.
\end{proof}

\begin{proposition}
If $m\ge 1$ and if $b>1$ then $RT(\mathcal{T},\psi_{RC})>1$ with positive ${\bf P}$-probability.
\end{proposition}

\begin{proof}
Recall the percolation $\mathcal{T}_C$ defined in \eqref{defpercocond}. By Proposition \ref{boundbrr}, we have that $br_r(\mathcal{T}_C)\ge b-\eps$ with positive probability. Let us denote $\Pi_C$ the set of all the cutsets in $\mathcal{T}_C$. On this event, we have that, if $b>1$,
\begin{equation}
\begin{split}
  \inf_{\pi \in \Pi } \sum_{e \in \pi} \left(\frac{1}{\sum_{i \le e} C^{-1}_i}\right)^{1+\eps}  &\ge   \inf_{\pi \in \Pi_C } \sum_{e \in \pi} \left(\frac{1}{\sum_{g \le e} C^{-1}_g}\right)^{1+\eps} \\
&  \ge\inf_{\pi \in \Pi_C } \sum_{e \in \pi} \left( |{e}|^{- 1-4\eps}\right)^{1+\eps}\\
&  \ge\inf_{\pi \in \Pi_C } \sum_{e \in \pi}  |{e}|^{- (b-2\eps)}>0,
\end{split}
\end{equation}
where we used \eqref{epsm2}. This implies that $RT(\mathcal{T}, \psi_{RC})>1$  with positive ${\bf P}$-probability, as defined in \eqref{defBR}. 
\end{proof}

\subsection{Recurrence: proof of the second item of Proposition \ref{mainth0}} \label{sec:recurrence}
We will again consider different cases and prove this time that $RT(\mathcal{T}, \psi_{RC})<1$, where we refer to \eqref{defBR} for a definition of this quantity.
\begin{proposition}
If $b\ge1$ and $bm<1$ then $RT(\mathcal{T}, \psi_{RC})<1$, ${\bf P}$-almost surely.
\end{proposition}

\begin{proof}
{Fix two positive parameters $\delta$ and $\eps$ such that $(1/m) - \delta>0$ and  
\begin{equation}\label{epsdelt} 
\left(\frac{1}{m} - \delta\right)(1 -\eps)\ge b + \delta.
\end{equation}
 The latter is possible as $mb<1$.

We  have that
\begin{equation}
\begin{split}
\P\left(\sum_{i \le e}  C_i^{-1} \le |e|^{\frac{1}{m} - \delta}\right)&\le \P\left(\bigcap_{i \le e}  C_i^{-1} \le |e|^{\frac{1}{m} - \delta}\right)\\
&= \left(1 - \frac {L\left(|e|^{\frac{1}{m}-\delta} \right)}{|e|^{(\frac{1}{m}- \delta) m}}\right)^{|e|} \le \exp\left \{ - |e|^{\delta m}L\left(|e|^{\frac{1}{m}-\delta} \right)\right \}.
\end{split}
\end{equation}
By the definition of branching-ruin number, there exists a sequence of cutsets $(\pi_n, n\geq 1)$ such that for any $n>0$, 
\begin{equation}\label{billburr}
\underset{e\in \pi_n}{\sum}\frac{1}{|e|^{b+\delta}}<\exp\{-n\}.
\end{equation}
On the other hand, for any $n>0$ we have,
\begin{equation}
\begin{split}
 \P\left(\bigcup_{e \in \pi_n} \Big\{\sum_{i \le e} C_i^{-1} \le |e|^{\frac{1}{m} - \delta}\Big\}\right) &\le  \sum_{e \in \pi_n} \P\left(\sum_{i \le e} C_i^{-1} \le |e|^{\frac{1}{m} - \delta}\right)
 \\
 &\le   \sum_{e \in \pi_n}\exp\left \{ - |e|^{\delta m}L\left(|e|^{\frac{1}{m}-\delta} \right)\right \}.
\end{split}
 \end{equation}
Note that there exists $n_0$ such that for any $n>n_0$, we have,
$$ \sum_{e \in \pi_n}\exp\left \{ - |e|^{\delta m}L\left(|e|^{\frac{1}{m}-\delta} \right)\right \} \leq \underset{e\in \pi_n}{\sum}\frac{1}{|e|^{b+\delta}}<\exp\{-n\} $$
Therefore, we have that
$$\underset{n\geq 1}{\sum}\, \P\left(\bigcup_{e \in \pi_n} \Big\{\sum_{i \le e} C_i^{-1} \le |e|^{\frac 1m - \delta}\Big\}\right) <\infty.$$
In virtue of the first Borel Cantelli Lemma, all edges $e\in \underset{n\geq 1}{\bigcup}\pi_n$, with the exception of finitely many,    satisty
\begin{equation}\label{rec1.5}
\sum_{i \le e} C_i^{-1} > |e|^{\frac{1}{m} - \delta}.
\end{equation}
Hence, for $n$ large enough
\begin{equation}\label{rec2}
\begin{aligned}
 \sum_{e \in \pi_n} \frac 1{(\sum_{i \le e} C_i^{-1})^{(1- \eps)}}&\le \sum_{e \in \pi_n}  \frac 1{|e|^{(\frac{1}{m} - \delta)(1 -\eps)}}\le  \sum_{e \in \pi_n}  \frac 1{|e|^{b+\delta}}<\exp\{-n\}.
\end{aligned}
\end{equation}
where we used  \eqref{epsdelt}. Hence,
\begin{equation}\label{manshit}
\underset{n\rightarrow \infty}{\lim} \sum_{e \in \pi_n} \frac 1{(\sum_{i \le e} C_i^{-1})^{(1- \eps)}}=0.
\end{equation}
Therefore, we have that 
\begin{equation}\label{rec1.6}
 0\leq \inf_{\pi \in \Pi} \sum_{e \in \pi}\left( \frac 1{\sum_{i \le e} C_i^{-1} }\right)^{1-\eps}\leq  \inf_{n\geq 1}\sum_{e \in \pi_n}\left( \frac 1{\sum_{i \le e} C_i^{-1} }\right)^{1-\eps}=0.
\end{equation}
Hence $RT(\mathcal{T}, \psi_{RC}) \le 1- \eps$.}
\end{proof}

The next result concludes the proof of Theorem \ref{mainth}.

\begin{proposition}
If $b<1$  then $RT(\mathcal{T}, \psi_{RC})<1$, ${\bf P}$-almost surely.
\end{proposition}

\begin{proof}
First, fix $\delta\in(0,1)$ such that
\begin{equation}
\label{equ:recurrence case2.0}
(1-\delta)^2 >b+\delta.
\end{equation}
The latter is possible as $b<1$. Then, note that, for any  $\eps \in (0,1)$, there exists $\eta>0$ such that
\begin{equation}
\P\left(C_0^{-1}>\eta\right)>1-\eps.
\end{equation}
{In the following, we denote $(C_j)_{j\ge0}$ a sequence conductances distributed like a generic conductance $C_e$. There exists a constant $c_{\delta,\eps}>0$ such that, for any $e\in E$,
\begin{equation}\label{equ:recurrencecase2.1}
\begin{split}
\P\left(\sum_{i \le |e|}  C_i^{-1} \le \eta|e|^{1- \delta}\right)&\le \P\left(\bigcup_{k=1}^{|e|/\lfloor |e|^{\delta}\rfloor} \bigcap_{j=(k-1)\lfloor |e|^{\delta}\rfloor+1}^{k\lfloor |e|^{\delta}\rfloor} \left\{  C_j^{-1}\le \eta \right\}\right)\\
&\le \frac{2}{1-\eps}|e|^{1-\delta}\P\left(C_0^{-1}\le\eta\right)^{|e|^\delta}\\
&\le \frac{2}{1-\eps} |e|^{1-\delta} \eps^{|e|^\delta}\\
&\le c_{\delta,\eps}|e|^{-b-\delta}.
\end{split}
\end{equation}
Indeed, to prove the first inequality above, note that
\begin{equation}
\begin{split}
\left\{\bigcup_{k=1}^{|e|/\lfloor |e|^{\delta}\rfloor} \bigcap_{j=(k-1)\lfloor |e|^{\delta}\rfloor+1}^{k\lfloor |e|^{\delta}\rfloor} \left\{  C_j^{-1}\le \eta \right\}\right\}^c&=\bigcap_{k=1}^{|e|/\lfloor |e|^{\delta}\rfloor} \bigcup_{j=(k-1)\lfloor |e|^{\delta}\rfloor+1}^{k\lfloor |e|^{\delta}\rfloor} \left\{  C_j^{-1}> \eta \right\}\\
&\subset  \left\{\sum_{i \le |e|}  C_i^{-1} > \eta|e|^{1- \delta}\right\}=  \left\{\sum_{i \le |e|}  C_i^{-1} \le \eta|e|^{1- \delta}\right\}^c.
\end{split}
\end{equation}
}
By the definition of branching-ruin number, there exists a sequence of cutsets $(\pi_n, n\geq 1)$ such that for any $n>0$, 
\begin{equation}\label{billburr2}
\underset{e\in \pi_n}{\sum}\frac{1}{|e|^{b+\delta}}<\frac{1}{c_{\delta,\eps}}\exp\{-n\}.
\end{equation}
We use \eqref{equ:recurrencecase2.1} and \eqref{billburr2} to obtain
\begin{equation}
\begin{split}
 \P\left(\bigcup_{e \in \pi_n} \Big\{\sum_{g \le e} C_g^{-1} \le \eta |e|^{1 - \delta}\Big\}\right) &\le c_{\delta,\eps} \sum_{e \in \pi_n} |e|^{-b-\delta}\le \exp(-n).
\end{split}
 \end{equation}
 Therefore, by Borel-Cantelli Lemma, as soon as $n$ is large enough, we have that
 \[
 \bigcap_{e \in \pi_n} \Big\{\sum_{i \le e} C_i^{-1} > \eta |e|^{1 - \delta}\Big\}
 \]
 holds, which implies that
 \begin{equation}\label{rec3}
\begin{aligned}
 \sum_{e \in \pi_n} \frac 1{(\sum_{i \le e} C_i^{-1})^{(1- \delta)}}&\le \frac{1}{\eta^{1-\delta}}\sum_{e \in \pi_n}  \frac 1{|e|^{(1 - \delta)(1 -\delta)}}\le  \frac{1}{\eta^{1-\delta}}\sum_{e \in \pi_n}  \frac 1{|e|^{b+\delta}}<\frac{\exp\{-n\}}{c_{\delta, \eps}{\eta}^{1-\delta}}.
\end{aligned},
\end{equation}
where we used \eqref{equ:recurrence case2.0}. Hence, following a strategy similar to \eqref{manshit}, \eqref{rec1.6}, we have that $RT(\mathcal{T}, \psi_{RC}) \le 1- \delta$, ${\bf P}$-almost surely. 
\end{proof}

\section{Proof of Theorem \ref{th:digging0} and  Lemma \ref{mainlemma}} \label{sec-digg}

In this section, we prove Lemma \ref{mainlemma}. With this in hand, Theorem \ref{th:digging} and Theorem \ref{th:digging0} will then trivially follow from Theorem \ref{maintheorem} (proved in Section \ref{sect_thdigg}) by noting that \eqref{condition1} is satisfied when $m_\nu=M\in\mathbb{N}$ for all $\nu\in V\setminus \{\r\}$.\\
{For any $e\in E$, we define
\begin{equation}
    \label{Psi}
    \Psi_{M, \lambda}(e)\Def \prod_{g\leq e}\psi_{{M,\lambda}}(g).
\end{equation}\\
As we will see in Section \ref{rubin}, $\Psi_{M, \lambda}(e)$ corresponds to the probability that ${\bf X}$, or $\widetilde{\bf X}$, when restricted to $[\r,e^+]$ and started from $\r$, hits $e^+$ before returning to $\r$.}

\begin{proof}[Proof of Lemma \ref{mainlemma}]
Here, we assume that $(m_\nu; \nu\in V)$ such that $m_\r=0$ and $m_\nu=M\in\mathbb{N}$ for all $\nu\in V\setminus \{\r\}$. Thus, by \eqref{psi} and \eqref{Psi}, we have that, if $\lambda\neq1$,
\begin{equation}
\Psi_{M,\lambda}(e)=\left(\frac{\lambda-1}{\lambda^{|e|}-1}\right)^{M+1},
\end{equation}
and, if $\lambda=1$,
\begin{equation}\label{eqlambda=1}
\Psi_{M, \lambda}(e)=|e|^{-M-1}.
\end{equation}
We will proceed by distinguishing a few cases.\\

\noindent
{\bf Case I: if $\lambda>1$ and $br(\mathcal T)<\lambda^{M+1}$.}\\
By \eqref{branchingdef},  there exists $\delta\in(0,1)$ such that
\begin{equation}\label{small}
\inf_{\pi\in\Pi} \sum_{e\in \Pi} \left(\lambda^{(M+1)(1-\delta)}\right)^{-|e|}=0.
\end{equation}
For any $\pi\in\Pi$, we have that
\begin{equation}\label{comp1}
\begin{split}
\sum_{e\in\pi} \Psi_{M, \lambda}(e)^{1-\delta}&=(\lambda-1)^{(M+1)(1-\delta)}\sum_{e\in\pi}\left(\frac{1}{\lambda^{|e|}-1}\right)^{(M+1)(1-\delta)}\\
&=(\lambda-1)^{(M+1)(1-\delta)}\sum_{e\in\pi}\frac{\lambda^{-|e|(M+1)(1-\delta)}}{(1-\lambda^{-|e|})^{(M+1)(1-\delta)}}\\
&\le\frac{(\lambda-1)^{(M+1)(1-\delta)}}{(1-\lambda^{-1})^{(M+1)(1-\delta)}}    \sum_{e\in\pi}\lambda^{-|e|(M+1)(1-\delta)}.
\end{split}
\end{equation}
Therefore, by \eqref{small},
\begin{equation}
\inf_{\pi\in\Pi} \sum_{e\in\pi} \Psi_{M, \lambda}(e)^{1-\delta}=0,
\end{equation}
which implies that $RT(\mathcal{T}, \psi_{M,\lambda})<1$.\\

\noindent
{\bf Case II: if $\lambda<1$ or if $\lambda>1$ and $br(\mathcal T)>\lambda^{M+1}$.}\\
Next, we prove that there exists $\delta>0$ and $\epsilon>0$ such that
\begin{equation}\label{small2}
\inf_{\pi\in\Pi} \sum_{e\in \Pi} \left(\lambda^{(M+1)(1+\delta)}\right)^{-|e|}>\epsilon.
\end{equation}
To prove the previous inequality, first note that this holds trivially if $\lambda<1$; second, if $\lambda>1$, we use the definition of the branching number and choose $\delta$ such that $\lambda^{(1+\delta)(M+1)}<br(\Tcal)$. A computation similar to \eqref{comp1} yields
\begin{equation}\label{comp2}
\begin{split}
\inf_{\pi\in\Pi}\sum_{e\in\pi} \Psi_{M, \lambda}(e)^{1+\delta}&\ge{(\lambda-1)^{(M+1)(1+\delta)}}\inf_{\pi\in\Pi}\sum_{e\in\pi}\lambda^{-|e|(M+1)(1+\delta)}\\
&>\epsilon.
\end{split}
\end{equation}
Therefore, we have that $RT(\mathcal{T},\psi_{M,\lambda})>1$.\\

\noindent
{\bf Case III:  $br_r(\mathcal{T})>M+1$ and $\lambda=1$.}\\
By \eqref{branchingdef2} , we have that there exists $\delta>0$ and $\epsilon>0$ such that
\begin{equation}
\inf_{\pi\in\Pi} \sum_{e\in \pi}|e|^{-(1+\delta)(M+1)}>\epsilon.
\end{equation}
Therefore, by \eqref{eqlambda=1}, we have that
\begin{equation}
\inf_{\pi\in\Pi} \sum_{e\in \pi} \left(\Psi_{M, \lambda}(e)\right)^{1+\delta}=\inf_{\pi\in\Pi} \sum_{e\in \pi}|e|^{-(1+\delta)(M+1)}>\epsilon, 
\end{equation}
which in turn implies that $RT(\mathcal{T},\psi_{M,\lambda})>1$.\\

\noindent
{\bf Case IV:  $br_r(\mathcal{T})<M+1$ and $\lambda=1$.}\\
We have that there exists $\delta>0$ such that
\begin{equation}
\inf_{\pi\in\Pi} \sum_{e\in \pi}|e|^{-(1-\delta)(M+1)}=0.
\end{equation}
Therefore, by \eqref{eqlambda=1}, we have that
\begin{equation}
\inf_{\pi\in\Pi} \sum_{e\in \pi} \left(\Psi_{M, \lambda}(e)\right)^{1-\delta}=\inf_{\pi\in\Pi} \sum_{e\in \pi}|e|^{-(1-\delta)(M+1)}=0.
\end{equation}
Therefore, we have that $RT(\mathcal{T},\psi_{M,\lambda})<1$.\\

\noindent
{\bf Case V: $br(\mathcal T)>\lambda^{M+1}$ and $\lambda=1$.}\\
Let us prove that $br(\mathcal T)>1$ implies that $br_r(\mathcal T)=\infty$, which gives the conclusion by Case III.
We have that there exists $\delta>0$ and $\epsilon>0$ such that
\begin{equation}\label{small2}
\inf_{\pi\in\Pi} \sum_{e\in \Pi} \left(1+\delta\right)^{-|e|}>\epsilon.
\end{equation}
Therefore, for any $\gamma>0$, there exists a constant $c_0>0$ depending only on $\gamma$, $\delta$ and $\epsilon$, such that
\begin{equation}\label{comp3}
\begin{split}
\sum_{e\in\pi} |e|^{-\gamma}&\ge c_0\sum_{e\in\pi} (1+\delta)^{-|e|} >c_0\epsilon.
\end{split}
\end{equation}
Taking the infimum over $\pi\in\Pi$ allows to conclude that $br_r(\mathcal{T})\ge\gamma$, for any $\gamma>0$, hence $br_r(\mathcal{T})=\infty$.
\end{proof}

\section{Extensions}
\label{rubin}
Here, we define the same construction as in \cite{CHK} and \cite{CKS}, which is a particular case of Rubin's construction. A large part of this section is a verbatim of Section 5 of  \cite{CKS}.\\
The following construction will allow us to emphasize useful independence properties of the walk on disjoint subsets of the tree.\\

Let $(\Omega, \mathcal{F},\bP)$ denote a probability space on which
\begin{align}\label{defY}
{\bf Y}=(Y(\nu,\mu,k): (\nu,\mu)\in V^2, \mbox{with }\nu \sim \mu, \textrm{ and }k \in \N)
\end{align}
is a family of independent  random variables, where $(\nu,\mu)$ {denotes} an {\it ordered} pair { of vertices}, and such that
\begin{itemize}
\item if {$\nu =\parent{\mu}$} and $k=0$, then $Y(\nu,\mu,0)$ a Gamma random variable with parameters $m_{\mu}+1$ and $1$;
\item otherwise, $Y(\nu,\mu,k)$ is an exponential random variable with mean $1$.
\end{itemize}
\begin{remark}\label{remarkgamma}
Recall that a Gamma random variable with parameters $m_{\mu}+1$ and $1$ has the same distribution as the sum of $m_{\mu}+1$\; i.i.d. exponential random variables with mean $1$.
\end{remark}
Below, we use these collections of random variables to generate the steps of $\widetilde{\bf X}$. Moreover,   we  define  a {\it family} of coupled walks using the same collection  of \lq clocks\rq\  $ {\bf Y}$.

Define, for any  $\nu,\mu\in V$ with $\nu\sim \mu$, the quantities
\begin{equation} \label{wj1}
r(\nu,\mu) \Def
{\lambda^{-|\nu|\vee |\mu|+1}}
\end{equation}

We are now going to define a family of coupled processes on the subtrees of $\mathcal{T}$. For any rooted subtree $\mathcal{T}'$ of $\mathcal{T}$, we define the {\it extension} $\widetilde{\bf X}^{ (\mathcal{T}')}=(V',E')$  on $\mathcal{T}'$ as follows. Let   the root $\r'$ of $\mathcal{T}'$ be defined as the vertex of $V'$ with smallest distance to $\r$.
For  a collection of nonnegative integers $\bar{k}=(k_\mu)_{\mu: [\nu,\mu]\in E'} $, let 
\[
A^{ (\mathcal{T}')}_{\bar{k},n,\nu}=\{\widetilde{X}^{ (\mathcal{T}')}_n = \nu\}\cap\bigcap_{\mu: [\nu,\mu]\in E'} \{\#\{1\le j \le n \colon (\widetilde{X}^{ (\mathcal{T}')}_{j-1},\widetilde{X}^{ (\mathcal{T}')}_j) = (\nu,\mu)\} = k_\mu\}.
\]
Note that the event $A^{ (\mathcal{T}')}_{\bar{k},n,\nu}$ deals with jumps along oriented edges.\\
Set $\widetilde{X}^{ (\mathcal{T}')}_0=\r'$ and, for $\nu$, $\nu'$ such that $[\nu, \nu']\in E'$ and for $n\ge0$, on the event 
\begin{align}\label{ursula}
A^{ (\mathcal{T}')}_{\bar{k},n,\nu}\cap \left\{\nu' = \argmin_{\mu: [\nu,\mu]\in E'}\Big\{\sum_{i=0}^{k_{\mu}}\frac{Y(\nu, \mu, i)}{r(\nu, \mu)} \Big\}\right\}, 
\end{align}
 we set $\widetilde{X}^{ (\mathcal{T}')}_{n+1} = \nu'$, where the function $r$ is defined in \eqref{wj1} and the clocks $Y$'s are from the same collection ${\bf Y}$ fixed in \eqref{defY}.\\
 
Thus, this defines $\widetilde{\bf X}^{(\mathcal{T})}$ as the extension on the whole tree.
  It is easy to check, from properties of independent exponential and Gamma  random variables, the memoryless property and Remark \ref{remarkgamma}, that this provides a construction of  $\widetilde{\bf X}$ on the tree $\mathcal{T}$.\\
This continuous-time embedding is classical: it is called {\it Rubin's construction}, after Herman Rubin  {(see the Appendix in  \cite{Dav90})}.\\
Now, if we consider proper subtrees $\mathcal{T}'$ of $\mathcal{T}$, one can check that, with these definitions, the steps of $\widetilde{\bf X}$ on the subtree $\mathcal{T}'$ are given by the steps of $\widetilde{\bf X}^{ (\mathcal{T}')}$ (see \cite{CHK} for details). As it was noticed in \cite{CHK}, for two subtrees $\mathcal{T}'$ and $\mathcal{T}''$ whose edge sets are disjoint, the extensions $\widetilde{\bf X}^{ (\mathcal{T}')}$ and $\widetilde{\bf X}^{ (\mathcal{T}'')}$ are independent as they are defined by two disjoint sub-collections of ${\bf Y}$.

Of particular interest will be the case where $\mathcal{T}'=[\r,\nu]$ is the unique self-avoiding path connecting $\r$ to $\nu$, for some $\nu\in\mathcal{T}$.  In this case, we write $\widetilde{\bf X}^{(\nu)}$ instead of $\widetilde{\bf X}^{([\r,\nu])}$, and we denote $T^{(\nu)}(\cdot)$ the return times associated to $\widetilde{\bf X}^{(\nu)}$. For simplicity, we will also write  $\widetilde{\bf X}^{(e)}$ and $T^{(e)}(\cdot)$ instead of $\widetilde{\bf X}^{(e^+)}$ and $T^{(e^+)}(\cdot)$ for $e\in E$.
Finally, it should be noted that, for any $e\in E$ and any $g\le e$,
\begin{align}\label{eqhit1}
\psi_{M, \lambda}(g)&=\bP\left(T^{(e)}(g^+)\circ \theta_{T^{(e)}(g^-)}<T^{(e)}({\r})\circ \theta_{T^{(e)}(g^-)}\right),\\ \label{eqhit}
\Psi_{M, \lambda}(e)&=\bP\left(T^{(e)}(e^+)<T^{(e)}({\r})\right),
\end{align}
where $\theta$ is the canonical shift on the trajectories.

\begin{remark} \label{diggonce}
Note that, for any vertex $\nu$, only the clocks $Y(\nu,\mu,0)$ with $\mu\sim\nu$, $\nu<\mu$, have a particular law. They follow a Gamma distribution instead of following an Exponential distribution. This resembles what would happen for a once-reinforced random walk (see \cite{CKS}). In this case, these clocks would still have an Exponential distribution but with a different parameter than the other ones (related to the reinforcement).\\
This means that an $M$-DRW$_\lambda$ is, in nature, very close to a once-reinforced random walk.
\end{remark}

%
%

\section{Proof of Theorem \ref{maintheorem}}\label{sect_thdigg}

In this section, we follow the blueprint of Section 7 of \cite{CKS}. In order to prove transience, the idea is to interpret the set of edges crossed before returning to $\r$ as the open edges in a certain correlated percolation.\\
A key step is to prove that this correlated percolation is {\it quasi-independent}, which will allow us to conclude its super-criticality from the super-criticality of some independent percolation.\\
Note that we will prove the transience of $\widetilde{\bf X}$ which is equivalent to the transience of ${\bf X}$.

\subsection{Link with percolation}
\label{linkwithperco}

Denote by $C(\r)$ the set of edges which are {crossed} by $\widetilde{\bf X}$ before returning to $\r$, that is: 
\begin{equation}
\mathcal C(\r)=\{e\in E: T(e^+)<T(\r) \}.
\end{equation}

This  set  can  be  seen  as  the  cluster  containing $\r$ in  some  correlated  percolation.  Next, we consider a different correlated percolation which will be more
convenient to us.  Recall Rubin'��s construction and the extensions introduced in Section~\ref{rubin}.  We define:
\begin{equation}
\mathcal C_{CP}(\r)=\{e\in E: T^{(e)}(e^+)<T^{(e)}(\r)\}.
\end{equation}

This defines a correlated percolation in which an edge $e\in E$ is open if $e\in \mathcal C_{CP}(\r)$.

\begin{lemma}
\label{lemma1}
We have that
\begin{equation}
\mathbb{P}(T(\r)=\infty)=\mathbb{P}(|\mathcal{C}(\r)|=\infty)=\mathbb{P}(|\mathcal{C}_{CP}(\r)|=\infty).
\end{equation}
\end{lemma}

\begin{proof}
We can follow line by line the proof of Lemma 11 in \cite{CKS}, except that one should replace  ${\bf X}$ by $\widetilde{\bf X}$.
\end{proof}

\subsection{Recurrence in Theorem \ref{maintheorem}: The case $RT(\mathcal{T},\psi_{M, \lambda})<1$}

The following result states the recurrence in Theorem \ref{maintheorem}.

\begin{proposition}[Proof of recurrence in Theorem \ref{maintheorem}: the case $RT(\mathcal{T},\psi_{M,\lambda})<1$]
If $RT(\mathcal{T},\psi_{M,\lambda})<1$ then ${\bf X}$ is recurrent.
\end{proposition}
\begin{proof}
This follows directly from Lemma \ref{lemma1} and Proposition \ref{propperc}.
\end{proof}

\subsection{Transience in Theorem \ref{maintheorem}: The case $RT(\mathcal{T},\psi_{M,\lambda})>1$}

Now, we want to prove the transience in Theorem \ref{maintheorem}. For this purpose, we need to check that the assumptions in Proposition \ref{propperc} are satisfied.\\

For simplicity, for a vertex $v\in V$, we write $v \in\C_{\mathrm{CP}}(\r)$ if one of the edges incident to $v$ is in $\C_{\mathrm{CP}}(\r)$. Besides, recall that for two edges $e_1$ and $e_2$, their common ancestor  with highest generation is the vertex denoted $e_1\wedge e_2$.

\begin{lemma}
\label{lemma2}
Assume that the condition \eqref{condition1} holds with some constant $M$. Then the correlated percolation induced by $\mathcal{C}_{CP}$ is quasi-independent, as defined in Definition \ref{defquasi}.
\end{lemma}

\begin{proof}
Here, we need to adapt the argument from the proof of Lemma 12 in \cite{CKS}.\\
Recall the construction of Section~\ref{rubin}. Note that if $e_1\wedge e_2=\varrho$, then the extensions on $[\varrho, e_1]$ and $[\varrho, e_2]$ are independent, then the conclusion of Lemma holds with $C=1$. 

Assume that $e_1\wedge e_2 \neq \varrho$, and  note  that  the  extensions  on $[\varrho, e_1]$ and $[\varrho, e_2]$ are dependent since they use the same  clocks on 
 $[\varrho, e_1\wedge e_2]$. Denote by $e$ the unique edge of $\mathcal{T}$ such that $e^+=e_1\wedge e_2$. We define the following quantities
 \begin{equation}
 \begin{split}
 N(e)&\Def\left|\left\{0\leq n \leq T^{(e)}(\varrho)\circ \theta_{T^{(e)}(e^+)}: (\widetilde{X}^{(e)}_n,\widetilde{X}^{(e)}_{n+1})=(e^+,e^-)\right\} \right|,\\
 L(e)&\Def \sum_{j=0}^{N(e)-1} \frac{Y(e^+,e^-,j)}{r(e^+,e^-)},
 \end{split}
 \end{equation}
 where $|A|$ denotes the cardinality of a set $A$ and $\theta$ is the canonical shift on trajectories. Note that $L(e)$ is the time consumed by the clocks attached to the oriented edge $(e^+,e^{-})$ before $\widetilde{\bf X}^{\ssup{e}}$, $\widetilde{X}^{\ssup{e_1}}$ or $\widetilde{X}^{\ssup{e_2}}$ goes back to $ {\r}$ once it has reached $e^+$. Recall that these three extensions are coupled and thus the time $L(e)$ is the same for the three of them.\\
 For $i\in \{1,2\}$, let $v_i$ be the vertex which is the offspring of $e^+$ lying the path from $\varrho$ to $e_i$. Note that $v_i$ could be equal to $e^+_i$. We define for $i\in \{1,2\}$: 
 \begin{equation}
 \begin{split}
 N^*(e_i)&=\left|\left\{0\leq n \leq T^{(e_i)}(e_i^+): (\widetilde{X}^{[e^+, e_i^+]}_n,\widetilde{X}^{[e^+,e_i^+]}_{n+1})=(e^+,v_i)\right\} \right|,\\
L^*(e_i)&=\sum_{j=0}^{N^*(e_i)-1}\frac{Y(e^+,e^-,j)}{r(e^+,e^-)}.
\end{split}
\end{equation}
 Here, $L^*(e_i)$, $i\in\{1,2\}$, is the time consumed by the clocks attached to the oriented edge $(e^+,v_i)$ before $\widetilde{\bf X}^{\ssup{e_i}}$, or $\widetilde{\bf X}^{[e^+,e_i^+]}$,   hits $e_i^+$.\\
Notice  that the three quantities $L(e)$, $L^*(e_1)$ and $L^*(e_2)$ are independent, and we also have: 
\begin{equation}
\{e_1, e_2\in \mathcal{C}_{CP}(\varrho)\}=\{T^{(e)}(e+)<T^{(e)}(\varrho)\} \cap \{L(e)>L^*(e_1)\} \cap \{L(e)>L^*(e_2)\}.
\end{equation}

 Now, conditioned on the event $\{T^{(e)}(e^+)<T^{(e)}(\varrho)\}$, the random variable $N(e)$ is simply a geometric random variable (counting the number of trials) with success probability $\lambda^{|e|-1}/\sum_{g\le e}\lambda^{|g|-1}$. The random variable $N(e)$ is independent of the family $Y(e^+,e^-,\cdot)$. As $Y(e^+,e^-,j)$ are independent exponential random variable for $j\ge0$, we then have that $L(e)$ is an exponential random variables with parameter 
 \begin{equation}\label{paramp}
 p:=\frac{\lambda^{|e|-1}}{\sum_{g\le e}\lambda^{|g|-1}}\times\lambda^{-|e|+1}=\frac{1}{\sum_{g\le e}\lambda^{|g|-1}}.
 \end{equation}
 A priori, $L^{*}(e_1)$ and $L^{*}(e_2)$ are not exponential random variable, but they have a continuous distribution. Denote $f_1$ and $f_2$ respectively the densities of $L^{*}(e_1)$ and $L^{*}(e_2)$. Then, we have that
 \begin{equation}
 \begin{split}
 \mathbb{P}\left(\left.e_1, e_2\in \mathcal C_{CP}(\varrho)\right|e_1\wedge e_2 \in \mathcal C_{CP}(\varrho)\right)&=\mathbb{P}\left(L(e)>L^*(e_1)\vee L^*(e_2)\right)\\
&={\int_{0}^{+\infty}\int_{0}^{+\infty}\int_{x_1\vee x_2}^{+\infty}p\, e^{-pt}f_1(x_1)f_2(x_2)dtdx_1dx_2}\\
&= \int_{0}^{+\infty}\int_{0}^{+\infty} e^{{-p}(x_1\vee x_2)}f_1(x_1)f_2(x_2)dx_1dx_2.\\
&\leq \int_{0}^{+\infty}\int_{0}^{+\infty} e^{\frac{-p}{2}(x_1+x_2)}f_1(x_1)f_2(x_2)dx_1dx_2.
 \end{split}
 \end{equation}
Thus, one can write
 \begin{equation}\label{qindep}
 \begin{aligned}
&  \mathbb{P}\left(\left.e_1, e_2\in \mathcal C_{CP}(\varrho)\right|e_1\wedge e_2 \in \mathcal C_{CP}(\varrho)\right)\\
&\leq \left (\int_{0}^{+\infty} e^{-px_1/2}f_1(x_1)dx_1\right )\cdot\left (\int_{0}^{+\infty} e^{-px_2/2}f_2(x_2)dx_2\right ).
\end{aligned}
 \end{equation}
 Note that, {for $i \in \{1, 2\}$},
 \begin{equation}
 \int_{0}^{+\infty} e^{-px_i/2}f_i(x_i)dx_i=\mathbb{P}\left(\widetilde{L}(e)>L^*(e_i)\right),
 \end{equation}
 where $\widetilde{L}(e)$ is an exponential variable with parameter $p/2$. Note that, in view of \eqref{paramp}, $\widetilde{L}(e)$ has the same law as $L(e)$ when we replace the weight of an edge $g'$ by $\lambda^{-|g'|+1}/2$ for $g'\leq e$ only, and keep the other weights the same. \\
 For $g\in E$ such that $e<g$, define the function $\widetilde{\psi}$ in a similar way as $\psi$, except that we replace the weight of an edge $g'$ by $\lambda^{-|g'|+1}/2$ for $g'\leq e$ only, and keep the other weights the same, that is, for $g\in E$, $e<g$,
$$
 \widetilde{\psi}_{{M,\lambda}}(g)=\left(\frac{2p^{-1}+\sum_{\nu:e<g'<g}\lambda^{|g'|-1}}{2p^{-1}+\sum_{\nu:e<g'\le g}\lambda^{|g'|-1}}\right)^{m_g+1}
$$
 We obtain: 
 \begin{equation}
 \begin{aligned}
& \mathbb{P}(\widetilde{L}(e)>L^*(e_1))=\prod_{g:e<g\leq e_1}\widetilde{\psi}(g)=\prod_{g:e<g\leq e_1}\left(\frac{2p^{-1}+\sum_{g':e<g'<g}\lambda^{|g'|-1}}{2p^{-1}+\sum_{g':e<g'\le g}\lambda^{|g'|-1}}\right)^{m_g+1}\\
  &= \mathbb{P}({L}(e)>L^*(e_1))\times\prod_{g:e<g\leq e_1}\left(1+\frac{p^{-1}}{p^{-1}+\sum_{g':e<g'< g}\lambda^{|g'|-1}}\right)^{m_g+1}\\
  &\qquad\qquad\qquad\qquad\qquad\times\left( 1-\frac{p^{-1}}{2p^{-1}+\sum_{g':e<g'\le g}\lambda^{|g'|-1}}\right)^{m_g+1}\\
   &= \mathbb{P}({L}(e)>L^*(e_1))\\
   &\qquad\times\prod_{g:e<g\leq e_1}\left(1+\frac{p^{-1}\lambda^{|g|-1}}{\left(p^{-1}+\sum_{g':e<g'< g}\lambda^{|g'|-1}\right)\left(2p^{-1}+\sum_{g':e<g'\le g}\lambda^{|g'|-1}\right)}\right)^{m_g+1}
   \end{aligned}
   \end{equation} 
   Hence,
   \begin{equation}
 \begin{aligned}
&\mathbb{P}(\widetilde{L}(e)>L^*(e_1))\\
 &\le \mathbb{P}({L}(e)>L^*(e_1))\\
   &\qquad\times\exp\left[(M+1)\sum_{g:e<g\leq e_1}\left(\frac{p^{-1}\lambda^{|g|-1}}{\left(p^{-1}+\sum_{g':e<g'< g}\lambda^{|g'|-1}\right)\left(p^{-1}+\sum_{g':e<g'\le g}\lambda^{|g'|-1}\right)}\right)\right]\\
   &\le \mathbb{P}({L}(e)>L^*(e_1))\exp\left[(M+1)\sum_{g:e<g\leq e_1}\left(\frac{p^{-1}\lambda^{|g|-1}}{\left(\sum_{g':g'< g}\lambda^{|g'|-1}\right)\left(\sum_{g':g'\le g}\lambda^{|g'|-1}\right)}\right)\right]\\
      &\le \mathbb{P}({L}(e)>L^*(e_1))\exp\left[(M+1)p^{-1}\sum_{g:e<g\leq e_1}\left(\frac{\sum_{g':g'\le g}\lambda^{|g'|-1}-\sum_{g':g'< g}\lambda^{|g'|-1}}{\left(\sum_{g':g'< g}\lambda^{|g'|-1}\right)\left(\sum_{g':g'\le g}\lambda^{|g'|-1}\right)}\right)\right]\\
         &\le \mathbb{P}({L}(e)>L^*(e_1))\exp\left[(M+1)p^{-1}\sum_{g:e<g\leq e_1}\left(\frac{1}{\sum_{g':g'< g}\lambda^{|g'|-1}}-\frac{1}{\sum_{g':g'\le g}\lambda^{|g'|-1}}\right)\right]\\
            &\le \mathbb{P}({L}(e)>L^*(e_1))\exp\left[(M+1)p^{-1}\left(\frac{1}{\sum_{g':g'\le e}\lambda^{|g'|-1}}-\frac{1}{\sum_{g':g'\le e_1}\lambda^{|g'|-1}}\right)\right]\\
     &\le\exp(M+1)\times \mathbb{P}({L}(e)>L^*(e_1)),
 \end{aligned}
 \end{equation}
 where we used condition \eqref{condition1}, the fact that we have a telescopic sum and where we used the definition \eqref{paramp} of $p$.\\
We have just proved that
\begin{equation}\label{qindep1}
\int_{0}^{+\infty} e^{-px_1/2}f_1(x_1)dx_1\leq {\exp\{M+1\}}\times \mathbb{P}(e_1\in \mathcal C_{CP}(\varrho)|e_1\wedge e_2 \in \mathcal C_{CP}(\varrho)).
\end{equation} 
By doing a very similar computation, one can prove that
\begin{equation}\label{qindep2}
\int_{0}^{+\infty} e^{-px_2/2}f_1(x_2)dx_2\leq  {\exp\{M+1\}}\times \mathbb{P}(e_2\in \mathcal C_{CP}(\varrho)|e_1\wedge e_2 \in \mathcal C_{CP}(\varrho)).
\end{equation}
The conclusion \eqref{qindep0} follows by using \eqref{qindep} together with \eqref{qindep1} and \eqref{qindep2}.
\end{proof}

\begin{proof}[Proof of transience in Theorem \ref{maintheorem}: The case $RT(\mathcal{T},\psi_{M,\lambda})>1$]
This follows directly from Lemma \ref{lemma1}, Lemma~\ref{lemma2} and Proposition  \ref{propperc}.
\end{proof}

\appendix

\section{Proof of Proposition \ref{propperc}}\label{appendix}

As above, we define a function $\Psi$ on the set of edges such that, for $e\in E$, 
\begin{equation}
\Psi(e)=\prod_{g\le e}\psi(e).
\end{equation}
By \eqref{gendefperco}, we have that \begin{equation}
\mathbb{P}\left[e\in\mathcal{C}(\r)\right]=\Psi(e).
\end{equation}
\subsection{Proof of Proposition~\ref{propperc} in the case $RT(\mathcal{T},\psi)<1$}

\begin{proposition}
If $RT(\mathcal{T},\psi)<1$, then a percolation such that \eqref{gendefperco} holds is subcritical. 
\end{proposition}

\begin{proof}
We use a first moment method.  For any cutset $\pi$, we have
$$\1_{\left\{|\mathcal{C}(\r)|=+\infty\right\}}\leq \sum_{e\in\pi}\1_{\left\{e\in \mathcal {C}(\r)\right\}} $$
and then 
$$\mathbb{P}\left[|\mathcal{C}(\r)|=+\infty\right]=\mathbb{E}\left[\1_{\left\{|\mathcal{C}(\r)|=+\infty\right\}}\right]\leq \sum_{e\in\pi}\mathbb{E}\left [\1_{\left\{e\in \mathcal {C}(\r)\right\}}\right]=\sum_{e\in\pi}\mathbb{P}\left[e\in \mathcal {C}(\r)\right] $$
Therefore 
$$\mathbb{P}\left[|\mathcal{C}(\r)|=+\infty\right]\leq \sum_{e\in \pi}\Psi(e).$$
Taking the infimum over $\pi\in \Pi$ allows to conclude that:
\begin{equation}
\label{equ:comparefirstmoment}
\mathbb{P}\left[|\mathcal{C}(\r)|=+\infty\right]\leq \inf_{\pi\in \Pi}\sum_{e\in \pi}\Psi(e).
\end{equation}

If $RT(\mathcal{T},\psi)<1$, the definition of $RT(\mathcal{T},\psi)$ (see \eqref{equ:RT}) implies that

 \begin{equation}
 \label{equ:infequa0}
\inf_{\pi\in \Pi}\sum_{e\in \pi}\Psi(e)=0
 \end{equation}
 We conclude the proof of proposition thanks to \eqref{equ:comparefirstmoment} and \eqref{equ:infequa0}.
\end{proof}

\bigskip
\subsection{Proof of  Proposition~\ref{propperc} in the case $RT(\mathcal{T},\psi)>1$}
As we are considering a quasi-independent percolation, we are able to lower-bound the probability of this correlated percolation to be infinite by the probability that some {\it independent} percolation is infinite. We do this by proving that a certain modified effective conductance is positive.\\

\begin{definition} \label{defmodcond1}
For any edge $e\in E$, let $c(e)=1$  if $|e|=1$ and, if $|e|>1$,   define the \emph{ adapted conductances}
\begin{align}\label{modcond}
{c}(e) = \frac 1{1- \psi(e)}   \Psi(e).
\end{align}
Define $\Ccal_{\rm eff}$ the effective conductance of $\mathcal{T}$ when  the conductance $c(e)$ is assigned to every edge $e\in E$.  For a definition of effective conductance, see \cite{LP} page 27.
\end{definition}


\begin{proposition}\label{proplowerboundperco}
Let $\mathcal{C}(\r)$ be the cluster of the root in a  percolation such that \eqref{gendefperco} holds. If the percolation is quasi-independent, then there exists $C_Q\in(0,\infty)$ such that
\[
\frac{1}{C_Q}\times\frac{\Ccal_{\rm eff}}{1+\Ccal_{\rm eff}} \le \bP(|\mathcal{C}( {\r})| = \infty). 
\]
\end{proposition}
\begin{proof}[Proof of Proposition \ref{proplowerboundperco}]
We can use the lower-bound  in Theorem  5.19 (page 145) of \cite{LP} to obtain the result.
\end{proof}

Recall that a flow $(\theta_e)$  on a tree is a nonnegative function on $E$ such that, for any $e\in E$, $\theta_e=\sum_{g\in E:g^-=e^+} \theta_g$. A flow is said to be a unit flow if moreover $\sum_{e:|e|=1}\theta_{e}=1$.\\
A usual technique in order to prove that some effective conductance is positive is to find a unit flow with finite energy. This is the content of the following statement, which is a simple consequence of classical results.

\begin{lemma}\label{simple}
Assume that \eqref{condition1} is satisfied.
Consider the tree $\mathcal{T}$ with the conductances defined in Definition \ref{defmodcond1} and assume that there exists a unit flow $(\theta_e)_{e\in E}$ on $\mathcal{T}$ from $ {\r}$ to infinity which has a finite energy, that is
\[
\sum_{e\in E}\frac{\left(\theta_e\right)^2}{c(e)}<\infty.
\]
Then, a quasi-independent percolation such that \eqref{gendefperco} holds is supercritical.
\end{lemma}
\begin{proof}
Using Proposition \ref{proplowerboundperco},  if $\C_{\rm eff}>0$ then  a quasi-independent percolation such that \eqref{gendefperco} holds is supercritical. By Theorem 2.11 (page 39) of \cite{LP},  $\C_{\rm eff}>0$ if and only if there exists a unit flow $(\theta_e)_{e\in E}$ on $\mathcal{T}$ from $ {\r}$ to infinity which has a finite energy.
\hfill
\end{proof}

The following result, from \cite{CKS}, is inspired by Corollary 4.2 of R.~Lyons \cite{L90}, which is itself a consequence of the max-flow min-cut Theorem. This result will provide us with a sufficient condition for the existence of a unit flow with finite energy.

\begin{proposition} \label{propLyons}
For any collection of positive numbers $(u_e)_{e\in E}$ such that $\sum_{e:|e|=1}u_{e}=1$ and
\begin{align}\label{condcor}
\inf_{\pi\in\Pi}\sum_{e\in\pi} u_{e}c(e)>0,
\end{align}
there exists  a nonzero flow whose energy is upper-bounded by
\[
\lim_{n\to\infty}\max_{e\in E: |e|=n}\sum_{g\le e} u_g.
\]
\end{proposition}

The proof is ended once we have proved the following proposition.

\begin{proposition}\label{proptrans}
If $RT(\mathcal{T},\psi)>1$, then a quasi independent percolation such that \eqref{gendefperco} holds is supercritical.
\end{proposition}

\begin{proof}
This proof follows line by line the proof of Proposition 18 in \cite{CKS}.\\
Fix a real number $\gamma\in \left(1,RT({\mathcal{T}},\psi)\right)$ and, for any edge $e\in E$, let us define $u_{e}=1$ if $|e|=1$ and, if $|e|>1$, 
\[
u_e=\left(1-\psi(e)\right)\prod_{g\le e}\left(\psi(g)\right)^{\gamma-1}.
\]
On one hand, we have that, for any $e\in E$,
\begin{align}\label{presque}
\sum_{g\le e}u_{g}\le C_\gamma.
\end{align}
Indeed, for each $e\in E$, we can apply Proposition 17 of \cite{CKS} to functions $f_e$ defined by $f_e(0)=1$ and, for $n\ge1$,  $f_e(n)=1-\psi(g)$ with $g$ the unique edge such that $g\le e$ and $|g|=n\wedge |e|$. We emphasize that \eqref{presque} holds with a uniform bound.\\
On the other hand, using \eqref{modcond}, we have
\begin{align*}
\inf_{\pi\in \Pi}\sum_{e\in\pi} u_{e}c(e) &= \inf_{\pi\in \Pi}\sum_{e\in\pi} \left(\left(1-\psi(e)\right)\left(\Psi(e)\right)^{\gamma-1}\right)\times\frac{\Psi(e)}{1-\psi(e)}\\
&= \inf_{\pi \in \Pi }\sum_{e\in\pi} \left(\Psi(e)\right)^{\gamma}>0.
\end{align*}
Proposition \ref{propLyons} and \eqref{presque} imply that there exists a nonzero flow $(\theta_e)$ whose energy is bounded as
\begin{align*}
\sum_{e\in E}\frac{\left(\theta_{e}\right)^2}{c(e)}\le \lim_{n\to\infty}\max_{e \in E: |e|=n}\sum_{g\le e} u_{g}\le C_\gamma.
\end{align*}
Therefore, there exists a unit flow with finite energy and Lemma \ref{simple} implies the result. 
\hfill
\end{proof}

%

\end{document}